\documentclass{amsart}
\usepackage{amsmath}
  \usepackage{paralist}
  \usepackage{graphics} %% add this and next lines if pictures should be in esp format
  \usepackage{epsfig} %For pictures: screened artwork should be set up with an 85 or 100 line screen
\usepackage{graphicx}  \usepackage{epstopdf}%This is to transfer .eps figure to .pdf figure; please compile your paper using PDFLeTex or PDFTeXify.
\usepackage[colorlinks=true]{hyperref}
\hypersetup{urlcolor=blue, citecolor=red}

\usepackage{amsfonts,amssymb}
\usepackage{wasysym}
\usepackage{verbatim}
\usepackage{psfrag}
\usepackage{graphics}
\usepackage{subfig}

\newtheorem{theorem}{Theorem}[section]

\newtheorem{lemma}[theorem]{Lemma}
\newtheorem{proposition}{Proposition}

\theoremstyle{definition}
\newtheorem{definition}[theorem]{Definition}
\newtheorem{remark}{Remark}

\mathchardef\emptyset="001F

\newcommand{\e}{\varepsilon}
\newcommand{\Om}{\Omega}

\newcommand{\weakst}{\stackrel{\ast}{\rightharpoonup}}

\newcommand{\R}{{\mathbb R}}

\renewcommand{\L}{{\mathcal L}}
\newcommand{\I}{{\mathcal I}}
\newcommand{\E}{{\mathcal E}}

\renewcommand{\O}{{\mathcal O}}

\newcommand{\A}{{\mathcal A}}

\newcommand{\Z}{{\mathbb Z}}
\newcommand{\N}{{\mathbb N}}

\newcommand{\M}{{\mathbb M}}
\renewcommand{\H}{{\mathcal H}}

\newcommand{\Mnn}{\M^{N\times N}}

\newcommand{\var}{\varphi}

\newcommand{\dist}{\mathrm{dist}}

\newcommand{\co}{{\rm co}}

\newcommand{\ut}{\tilde u}
\newcommand{\w}{\mathrm{w}}
\newcommand{\wu}{\mathrm{v}}
\renewcommand{\u}{\mathrm{u}}

\newcommand{\fcc}[1]{^{\mathrm{F}#1}}
\newcommand{\hcp}[1]{^{\mathrm{H}#1}}
\newcommand{\bcc}[1]{^{\mathrm{B}#1}}
\newcommand{\dc}[1]{^{\mathrm{D}#1}}
\newcommand{\honey}{^{\text{\davidsstar}}} % ^{\text{\varhexagon}}
\newcommand{\ga}{\gamma}
\newcommand{\Ga}{\Gamma}
\newcommand{\LL}{\L}
\newcommand{\TT}{\mathcal T}
\newcommand{\NN}[2]{#1,\,#2\ {\rm NN}} %{#1\stackrel{\textnormal{NN}}{\sim}#2}
\newcommand{\NNN}[2]{#1,\,#2\ {\rm NNN}}

\newcommand{\be}{\begin{equation}}
\newcommand{\ee}{\end{equation}}
\newcommand{\bes}{\begin{equation*}}
\newcommand{\ees}{\end{equation*}}
\newcommand{\bea}{\begin{eqnarray}}
\newcommand{\eea}{\end{eqnarray}}
\newcommand{\beas}{\begin{eqnarray*}}
\newcommand{\eeas}{\end{eqnarray*}}
\renewcommand{\d}{\mathrm{d}}
\newcommand{\D}{\nabla} %{\mathrm{D}}

\newcommand{\eps}{\varepsilon}

\renewcommand{\mod}[1]{\left|#1\right|}

\title[Rigidity of three-dimensional lattices] %Use the shortened version of the full title
      {Rigidity of three-dimensional lattices \\ and dimension reduction in \\ heterogeneous nanowires}

\author[G. Lazzaroni, M. Palombaro and A. Schl\"omerkemper]{}

 \email{giuliano.lazzaroni@sissa.it}
 \email{M.Palombaro@sussex.ac.uk}
 \email{anja.schloemerkemper@mathematik.uni-wuerzburg.de}

\begin{document}
\maketitle

% Enter the first author's name and address:
\centerline{\scshape Giuliano Lazzaroni}
\medskip
{\footnotesize
% please put the address of the first author
 \centerline{SISSA}
   \centerline{Via Bonomea 265}
   \centerline{34136 Trieste, Italy}
   %\centerline{Italy}
} % Do not forget to end the {\footnotesize by the sign }

\medskip

\centerline{\scshape Mariapia Palombaro}
\medskip
{\footnotesize
 % please put the address of the second  and third author
 \centerline{University of Sussex}
   \centerline{Department of Mathematics}
   \centerline{Pevensey 2 Building}
   \centerline{Falmer Campus}
   \centerline{Brighton BN1 9QH, United Kingdom}
   %\centerline{United Kingdom}
 }
 \medskip

\centerline{\scshape Anja Schl\"omerkemper}
\medskip
{\footnotesize
  \centerline{University of W\"urzburg}
   \centerline{Institute of Mathematics}
   \centerline{Emil-Fischer-Stra\ss{}e 40}
   \centerline{97074 W\"urzburg, Germany}
 %  \centerline{Germany}
}

\bigskip

% The name of the associate editor will be entered by an editorial staff
% "Communicated by the associate editor name" is not needed for special issue.
% \centerline{(Communicated by the associate editor name)}

%The abstract of your paper
\begin{abstract}
In the context of nanowire heterostructures we perform a discrete to continuum limit of the corresponding free energy by means of $\Ga$-convergence techniques. Nearest neighbours are identified by employing the notions of Voronoi diagrams and Delaunay triangulations. The scaling of the nanowire is done in such a way that we perform not only a continuum limit but a dimension reduction simultaneously. The main part of the proof is a discrete geometric rigidity result that we announced in an earlier work and show here in detail for a variety of three-dimensional lattices. We perform the passage from discrete to continuum twice: once for a system that compensates a lattice mismatch between two parts of the heterogeneous nanowire without defects and once for a system that creates dislocations. It turns out that we can verify the experimentally observed fact that the nanowires  show dislocations when the radius of the specimen is large.
\end{abstract}

\section*{Introduction}
Rigidity results of elastic materials have been of great interest in mathematical continuum mechanics in recent years, in particular since the seminal work by Friesecke, James and M\"uller \cite{fjm}.
Such results yield a deeper insight into the properties  of materials through an estimate of the distance of the deformation gradient from the set of rotations; this distance is in turn estimated from above by the free energy of the system.

The rigidity estimates turn out to be crucial steps in various proofs as for instance of $\Gamma$-convergence results in the context of dimension reduction. This was also the case in our earlier paper \cite{LaPaSc2014}, in which we derived a discrete to continuum limit and a dimension reduction of an energy of a heterogeneous nanowire (see \cite{LPSpro} for an abridged version). There we presented a detailed analysis of the passage from the two-dimensional setting to the one-dimensional limit, and we gave a summary of the corresponding dimension reduction from three dimensions to one dimension. The purpose of this article is to show the rigidity estimates (Section~\ref{sec:rigidity}) and the main features of the latter case in detail.

Further, we elaborate on various three-dimensional lattices that are of importance in applications: the face-centred cubic lattice, the hexagonal close-packed, the body-centred cubic lattice and the diamond cubic lattice, see Section~\ref{3d}. These lattices occur for instance in aluminum and gold,  magnesium and zinc, iron and tungsten, and germanium and silicon, respectively. Note that Si/Ge nanowires have applications in the semiconductor optoelectronics \cite{Kavanagh,Schmidtetal2010}.
In Section~\ref{sec:rigidity} we show that our discrete rigidity result applies to all these lattice structures.
The main property of such lattices is their geometric rigidity: they define a tessellation of the space into rigid polyhedra whose edges correspond to bonds in the lattice.
Our approach does not work in non-rigid lattices, like a simple cubic crystal with nearest-neighbour interactions only.
\par
We are interested in the mathematical modeling of dislocations in heterogeneous nano\-wires.
We assume that the material consists of two parts with the same lattice structure but different lattice constants. The interface between the two parts is assumed to be flat. The material overcomes the lattice mismatch either defect-free or by creating dislocations. As was pointed out by Ertekin et al. \cite{egcs}, it is the radius of the nanowire which determines whether the material creates dislocations or is defect free. In our model the radius roughly corresponds to the number of layers of atoms parallel to the direction of the wire, see Section \ref{sec:not}.
We prove that it is energetically more favourite to create dislocations than to relieve the mismatch in a defect-free way if the thickness of the nanowire is sufficiently large (see Remark~\ref{finalrem}).

The underlying idea of our mathematical model, which we introduce in Section~\ref{sec:not} in detail, goes back to the variational model proposed in  \cite{mp} in the context of nonlinear elasticity
%\cite{egcs} in the context of linearized elasticity, which was later generalized to nonlinear elasticity in \cite{mp}
and which was later generalized
 to a discrete to continuum setting in \cite{LaPaSc2014}. As before we assume that the total energy only consists of nearest-neighbour interactions which are harmonic, though it is possible to generalize this as discussed in \cite[Section~4]{LaPaSc2014}. In order to be able to apply a rigidity estimate, we always impose a non-interpenetration condition, which ensures that the deformations of the discrete setting preserve the orientation of each cell; similar assumptions were made e.g.\ in \cite{BSV,FT}. The non-interpenetration assumption can be dropped if one takes into account interactions beyond nearest neighbours, see the recent work \cite{ALP}. It is worth mentioning that a related variational model for misfit dislocations has been recently proposed in \cite{FPP}.

As in \cite{LaPaSc2014} we distinguish the systems with and without defects already in the given reference configuration. For both such systems we study the corresponding free energy of nearest- neighbour interactions in a discrete to continuum limit with dimension reduction. For the definition of the nearest neighbours in the discrete settings close to the interface it is useful to work with the notion of Delaunay triangulations and Voronoi cells, see \cite{LaPaSc2014}, where this was introduced for the first time to describe configurations with dislocations, see also Section~\ref{3d} for an introduction.

In Section~\ref{sec:4} we compare the minimizers of the limiting functionals, which characterize the minimum cost needed to compensate the lattice mismatch with and without defects, respectively.  It turns out that this cost depends on the thickness of the wire described by a mesoscale parameter $k$. More precisely, it depends quadratically on $k$ if there are dislocations, and scales like  $k^3$ if there are no defects. Hence for sufficiently large $k$, i.e. large radius of the wire, dislocations are energetically preferred. The result is based on a scaling argument. In particular for applications in semiconductor optoelectronics it would be interesting to know the threshold $k_c$ below which the nanowire deforms defect-free. This is however out of reach with our current methods so that we leave this as an open problem for future research.

\par
\section{Three-dimensional lattices}\label{3d}
We consider various three dimensional lattices whose unit cells are rigid convex polyhedra.
In this context, rigidity is understood in the following sense:
once the lengths of the edges of a polyhedron are given,
then the polyhedron is determined up to rotations and translations,
under the assumption that the polyhedron itself is convex.
We recall that a convex polyhedron is rigid if and only if its facets are triangles,
according to the classical Cauchy Rigidity Theorem (see, e.g., \cite{handbook-geometry}).
We consider four types of discrete lattices in dimension three:
the face-centred cubic, the hexagonal close-packed, the body-centred cubic,
and the diamond cubic. They should be interpreted as prototypes to which our approach can be applied,
under slight modifications in each case. For a general overview on lattice structures see, e.g., \cite{Grosso}.

\par
All the lattices we will introduce, fulfil a property of rigidity.
Indeed, the corresponding nearest-neighbour bonds provide a tessellation of the space
into rigid convex polyhedra, as we will make precise case by case.
(In the diamond cubic, also next-to-nearest neighbours will be used.)
We always assume a non-interpenetration condition, see \eqref{ad-3d} below.
\par
A major role in modeling is then played by the choice of the
nearest neighbours of each lattice.
Here they are defined according to the notion of
\emph{Delaunay pretriangulation}, as given in the following definitions.
Such a general definition can be applied also when the lattice is irregular,
so in particular across the interface between the phases, see Section \ref{subsec:setting}.
\par
For later convenience we give the definition for all dimensions $N\ge2$.
Let $\LL\subset\R^N$ be a countable set of points such that there exist $R,r>0$ with
$\inf_{x\in\R^N} \# \big(\LL\cap B(x,R)\big)\ge1$ and
$|x-y|\ge r$ for every $x,y\in\LL$, $x\neq y$,
where $B(x,R):=\{y\in\R^N\colon |x-y|<R\}$.
\begin{definition}[Voronoi cells] \label{def:Voronoi}
The Voronoi cell of a point $x\in\LL$ is the set
$$
C(x):= \{ z\in\R^N\colon |z-x|\le|z-y| \ \forall\, y\in\LL \} \,.
$$
The Voronoi diagram associated with $\LL$ is the partition $\{C(x)\}_{x\in\LL}$.
\end{definition}
\begin{definition}[Delaunay pretriangulation]\label{def:Del-pre}
The Delaunay pretriangulation associated with $\LL$ is a partition of $\R^N$ in open nonempty hyperpolyhedra with vertices in $\LL$,
such that two points $x,y\in\LL$ are vertices of the same hyperpolyhedra if and only if $C(x)\cap C(y)\neq\emptyset$.
\end{definition}
\begin{definition}[Nearest neighbours]\label{def:NN}
Two points $x,y\in\LL$, $x\neq y$, are said to be nearest neighbours (and we write: $\NN{x}{y}$)
if they are vertices of an edge of one of the hyperpolyhedra of the Delaunay pretriangulation.
\end{definition}
\begin{definition}[Next-to-nearest neighbours]\label{def:NNN}

Two points $x,y\in\LL$, $x\neq y$, are said to be next-to-nearest neighbours
(and we write: $\NNN{x}{y}$)
if, setting
$$
\LL_*(x):=\LL\setminus\{y\colon \NN{x}{y}\}\,,
$$
we find that  $\H^{N-1}(C^x_*(x)\cap C^x_*(y))>0$, where $\{C^x_*(y)\}_{y\in\LL_*(x)}$ is
the Voronoi diagram associated with $\LL_*(x)$.
%The next-to-nearest neighbours of a point $x\in\LL$ are defined as follows:
%first we drop its nearest neighbours from the lattice and set
%$$
%\LL_*(x):=\LL\setminus\{y\colon \NN{x}{y}\} \,,
%$$
%so we define the Voronoi diagram $\{C^x_*(y)\}_{y\in\LL_*(x)}$ associated with $\LL_*(x)$.
%Then, we define the next-to-nearest neighbours of $x$ as the points $y\in\LL_*(x)$, $y\neq x$,
%such that $\H^2(C^x_*(x)\cap C^x_*(y))>0$, i.e., the points share a facet;
%in this case we write $\NNN{x}{y}$.
\end{definition}
The Voronoi diagram and the Delaunay pretriangulation
associated with a lattice are unique.
For these and other properties we refer to \cite[Section 1]{LaPaSc2014} and references therein.
\subsection{FCC lattice}
The face-centred cubic lattice is the typical structure of metals such as aluminium, gold, nickel, and platinum.
It is the Bravais lattice generated by the vectors
$$
\wu\fcc{}_1:=\sqrt2(1,0,0) \,, \quad
\wu\fcc{}_2:=\sqrt2\big(\tfrac12,\tfrac12,0\big) \,, \quad
\wu\fcc{}_3:=\sqrt2\big(0,\tfrac12,\tfrac12\big) \,,
$$
namely
$$
\L\fcc{}:= \{ \xi_1\wu\fcc{}_1+\xi_2\wu\fcc{}_2+\xi_3\wu\fcc{}_3 \colon \ \xi_1,\xi_2,\xi_3\in\Z  \}.
$$
The resulting lattice is obtained by repeating  periodically in the space
a cubic cell of side $\sqrt2$,
where the atoms lie at the vertices and at the centre of each facet.
It is readily seen that two points $x,y\in\L\fcc{}$ are nearest neighbours in the sense of Definition \ref{def:NN}
if and only if $\mod{x-y}=1$, i.e.,
they are joined by half a diagonal of a facet of the cubic cell.
Each atom has twelve nearest neighbours.
The Delaunay pretriangulation provides a subdivision of the space
into regular tetrahedra and octahedra of side one, thus in rigid convex polyhedra,
see Figure \ref{fig:fcc}.
Remark that the diagonals of the octahedra, whose length is $\sqrt2$,
correspond to next-to-nearest neighbours. The latter will not enter the definition of the energy
\eqref{en:fcc}.
\begin{figure}[p]
%\begin{figure}[ht]
\centering
\subfloat[]{
\includegraphics[width=.24\textwidth]{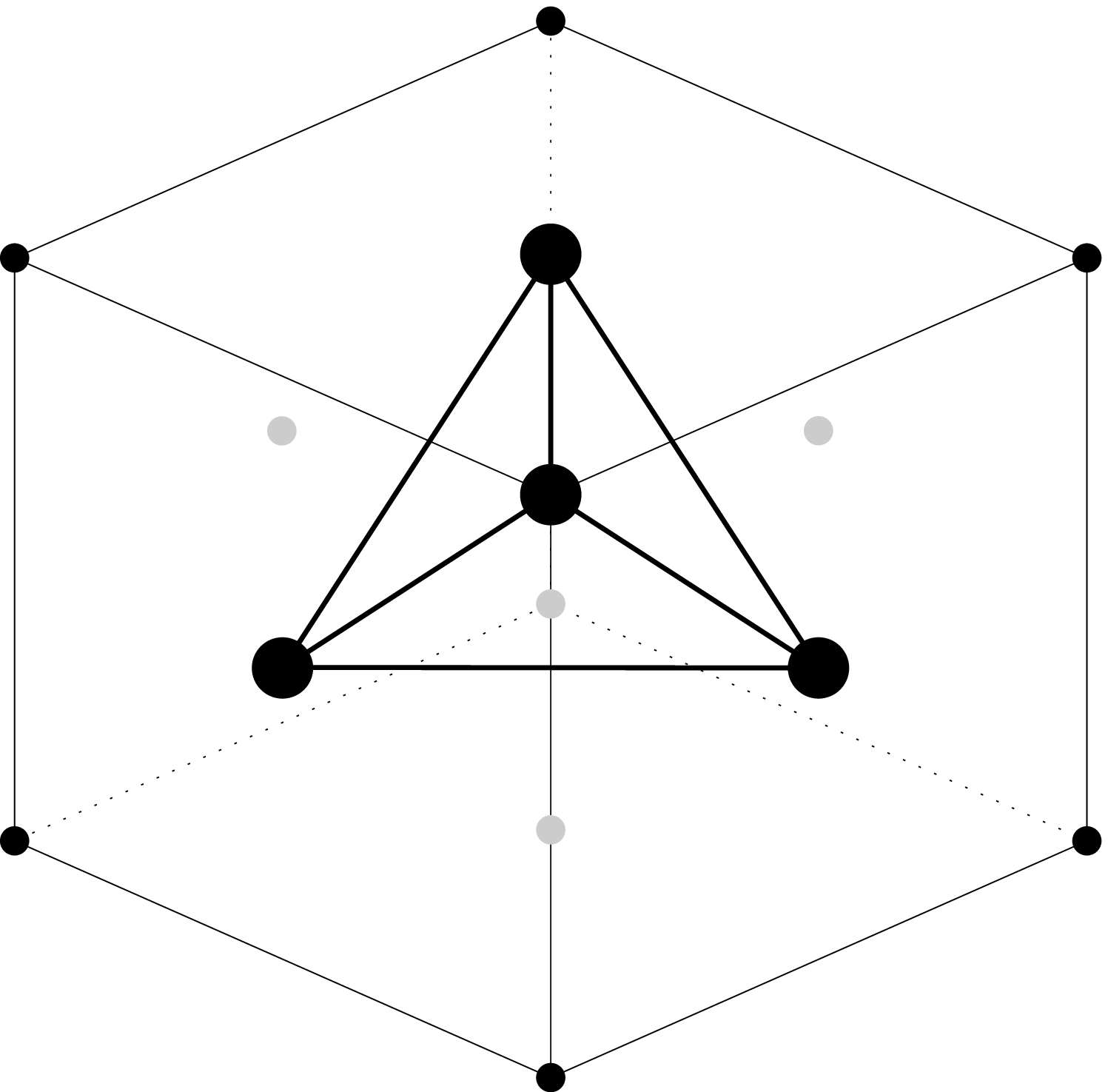}
\label{fig:subfig1}
}
\hfill
\subfloat[]{
\includegraphics[width=.24\textwidth]{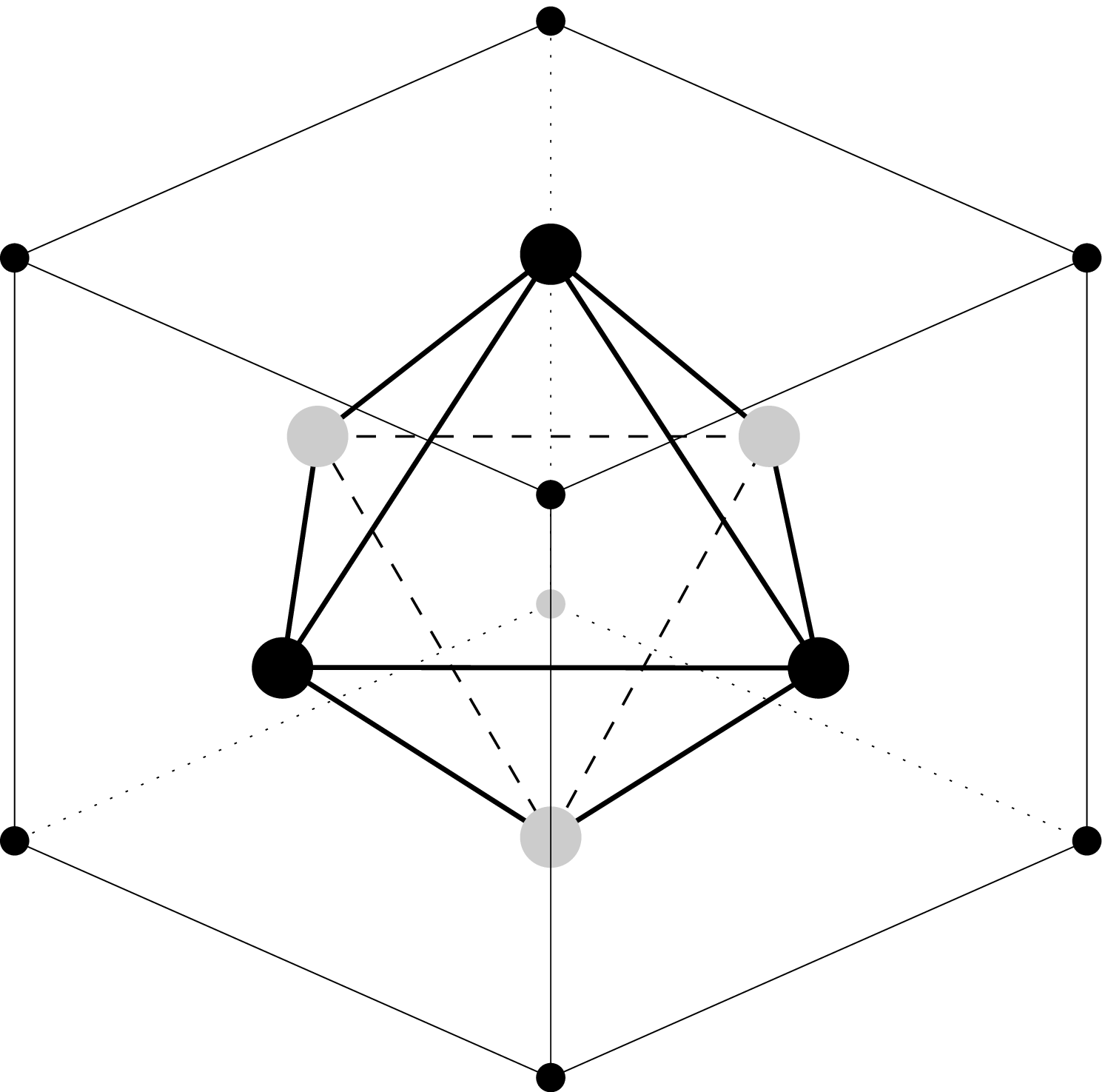}
\label{fig:subfig3}
}
\hfill
\subfloat[]{
\includegraphics[width=.24\textwidth]{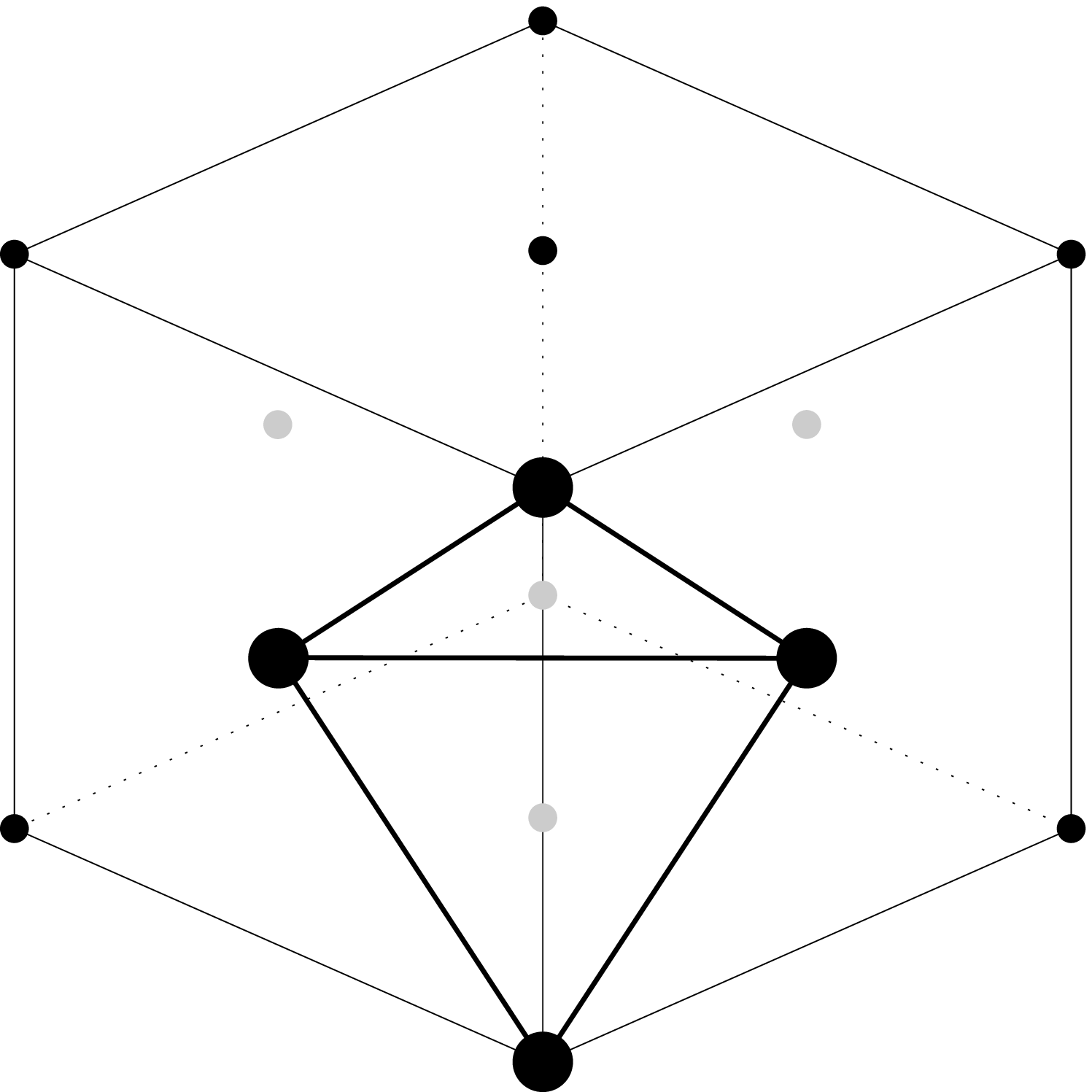}
\label{fig:subfig2}
}
\caption{In the face-centred cubic lattice
the nearest-neighbour structure of the atoms provides a subdivision of the space into tetrahedra (a) and octahedra (b). Figure~(c) shows a quarter of an octahedron in the same unit cell.
Grey dots denote points lying on the hidden facets.
}\label{fig:fcc}
\end{figure}
\subsection{HCP lattice}
Our approach works also for non-Bravais lattices such as
the hexagonal close-packed structure found in some metals as, e.g., magnesium and zinc.
It is defined by
$$
\L\hcp{}:=\{\u\hcp{}_i +  \xi_1\wu\hcp{}_1+\xi_2\wu\hcp{}_2+\xi_3\wu\hcp{}_3 \colon \ \xi_1,\xi_2,\xi_3\in\Z \,, \ i=1,2\} \,,
$$
where
$$
\wu\hcp{}_1:=\big(0,0,\tfrac{2\sqrt{6}}{3}\big)\,, \quad \wu\hcp{}_2:=\big(\tfrac{1}{2},\tfrac{\sqrt{3}}{2},0\big)\,,\quad
\wu\hcp{}_3:=\big(-\tfrac{1}{2},\tfrac{\sqrt{3}}{2},0\big)
$$
are generators of two sublattices and
$$
\u\hcp{}_1:=(0,0,0)\,,\quad
\u\hcp{}_2:=\big(0,\tfrac{\sqrt{3}}{3}, \tfrac{\sqrt{6}}{3}\big)
$$
are called vectors of the basis.
The lattice is thus obtained by merging two Bravais sublattices (defined for $i=1$ and $i=2$, respectively).
As in the previous case, the nearest neighbours are those couples with distance one,
each atom has twelve nearest neighbours,
and the Delaunay pretriangulation consists of regular tetrahedra and octahedra of side one,
see Figure \ref{fig:hcp}. As before, the diagonals of the octahedra, which correspond to
next-to-nearest neighbour interactions, will not enter the definition of the energy \eqref{en:hcp}.
\begin{figure}[p]
%\begin{figure}[ht]
\centering
\subfloat[]{
\includegraphics[width=.28\textwidth]{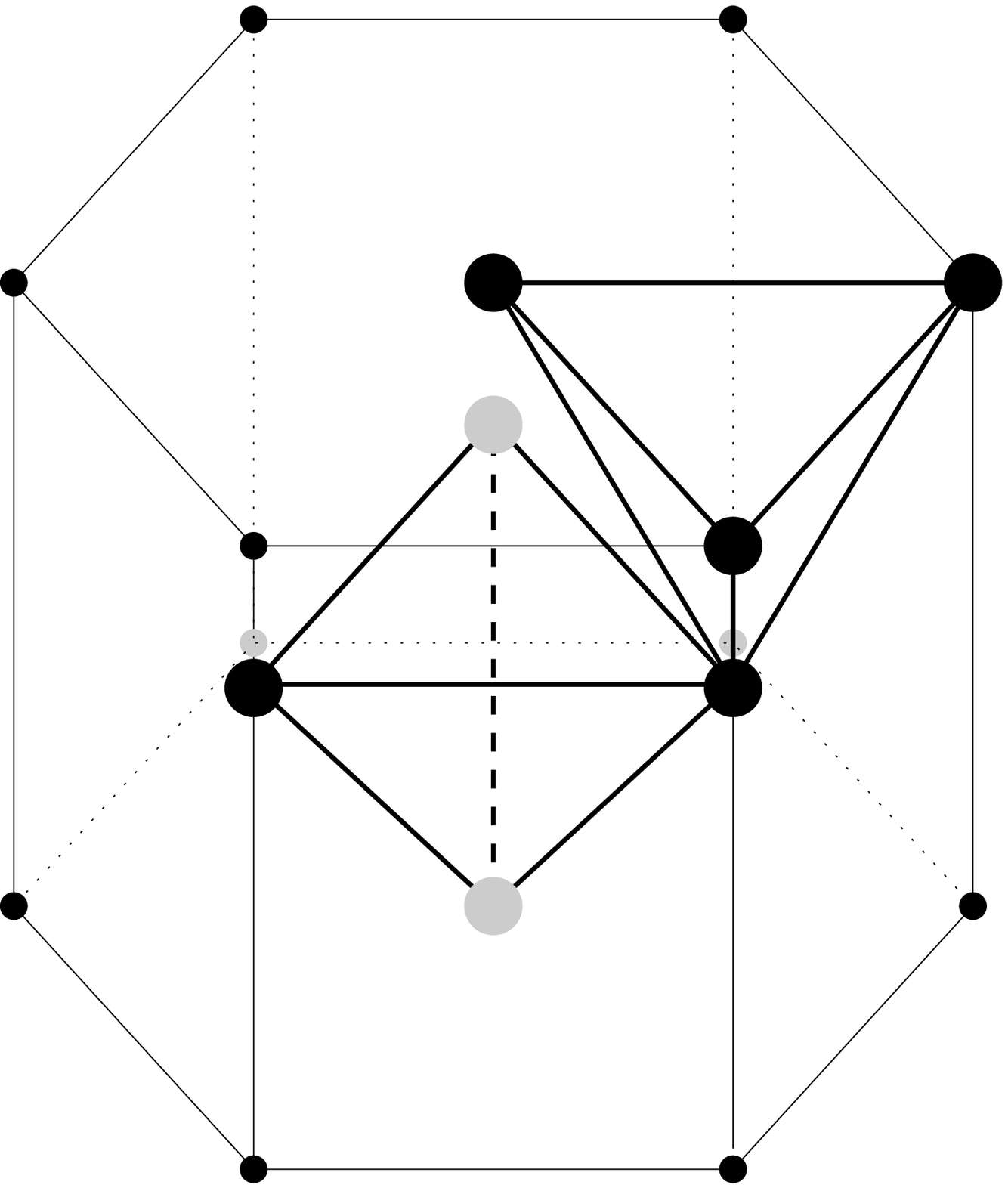}
}
\hspace{.15\textwidth}
\subfloat[]{
\includegraphics[width=.28\textwidth]{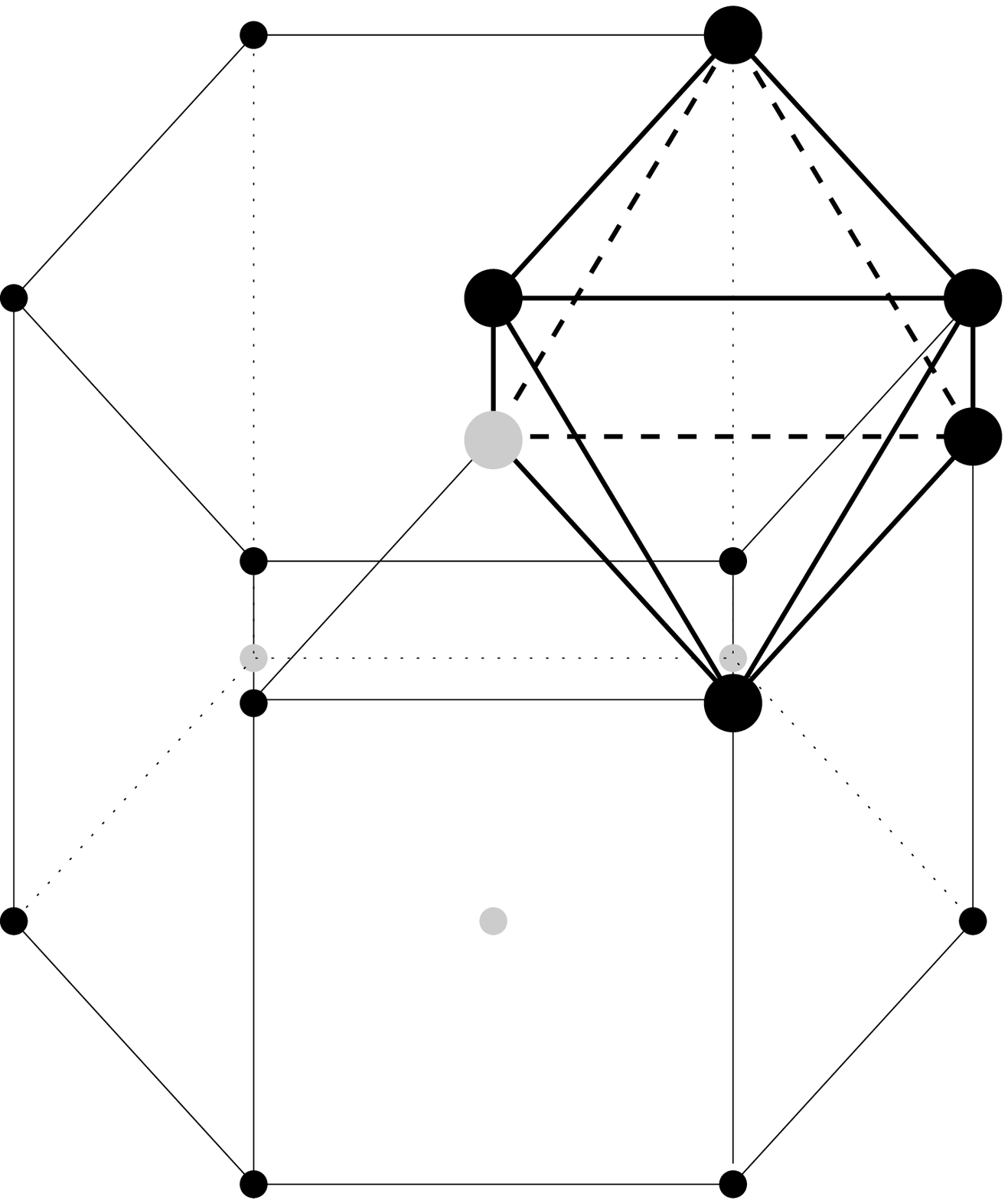}
}
\caption{The hexagonal close-packed lattice is associated with
a tessellation of tetrahedra and octahedra as the ones in the figure.
Only some of the bonds and some of the polyhedra of the pretriangulation are displayed.
}\label{fig:hcp}
\end{figure}
\subsection{BCC lattice} 
The body-centred cubic lattice 
\begin{figure}[p]
\centering
\includegraphics[width=.45\textwidth]{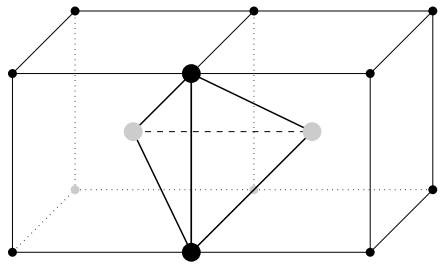}
\caption{The body-centred cubic lattice is associated with
a tessellation of irregular tetrahedra as the one in the figure.
}\label{fig:bcc}
\end{figure}
is typical of some metals as, e.g., iron and tungsten.
It is the Bravais lattice generated by the vectors
$$
\wu\bcc{}_1:=\tfrac{\sqrt2}{2}(-1,1,1) \,, \quad
\wu\bcc{}_2:=\tfrac{\sqrt2}{2}(1,-1,1) \,, \quad
\wu\bcc{}_3:=\tfrac{\sqrt2}{2}(1,1,-1) \,,
$$
namely
$$
\L\bcc{}:= \{ \xi_1\wu\bcc{}_1+\xi_2\wu\bcc{}_2+\xi_3\wu\bcc{}_3 \colon \ \xi_1,\xi_2,\xi_3\in\Z  \} \,.
$$
The resulting lattice can be viewed by repeating periodically in the space a cubic cell of side $\sqrt2$,
where the atoms lie at the vertices and at the centre of the cube.
According to Definition \ref{def:NN},
the nearest neighbours are those couples with distance $\frac{\sqrt6}{2}$,
(i.e., those joined by half a diagonal of the cubic cell,)
as well as those couples with distance $\sqrt2$
(i.e., those joined by an edge of the cubic cell).
Thus, in contrast with the face-centred cubic,
in this case the notion of nearest neighbours differs
from other notions based on the Euclidean distance.
According to this definition each atom has 14 nearest neighbours.
Correspondingly, the Delaunay pretriangulation consists
in a subdivision of the space into irregular tetrahedra,
with four edges of length $\frac{\sqrt6}{2}$ and two of length $\sqrt2$,
see Figure \ref{fig:bcc}.
Such an asymmetry in the definition of nearest neighbours leads to consider an anisotropic energy, see \eqref{bcc-energy}.
\subsection{DC lattice}
Finally, we present the diamond cubic lattice, which is composed
of two interpenetrating face-centred cubic lattices (thus, it is non-Bravais).
It is relevant in applications to nanowires, since it is the structure
of materials of use, such as silicon and germanium \cite{Kavanagh}.
When the sites of the two interpenetrating lattices are filled with
two different species of atoms, the structure is called zincblende
and is typical of Gallium arsenide (GaAs) and Indium arsenide (InAs),
also used in technical applications to semiconductor optoelectronics \cite{Kavanagh}.
\par
The diamond cubic structure is defined by
$$ %\vspace{.3em}
\L\dc{}:=\{\u_i\dc{} +  \xi_1\wu\dc{}_1+\xi_2\wu\dc{}_2+\xi_3 \wu\dc{}_3 \colon \ \xi_1,\xi_2,\xi_3\in\Z \,, \ i=1,2\} \,, %\`vspace{.3em}
$$

where $\wu_j\dc{}:=\wu_j\fcc{}$, $j=1,2,3$, are as in the face-centred cubic and
$$
\u_1\dc{}:=(0,0,0)\,,\
\u_2\dc{}:=\sqrt2\big(\tfrac14,\tfrac14,\tfrac14\big)
$$
compose the basis. It is convenient to split the lattice as follows,
\be\label{def:subdc} %\vspace{.3em}
\L\dc{}= \L\dc{_1} \cup \L\dc{_2} \,, %\vspace{.3em}
\ee
$$
\L\dc{_i}:=\{\u_i\dc{} +  \xi_1\wu\dc{}_1+\xi_2\wu\dc{}_2+\xi_3 \wu\dc{}_3 \colon \ \xi_1,\xi_2,\xi_3\in\Z\} \,, \ i=1,2 \,, %\vspace{.3em}
$$
where the sublattices $\L\dc{_i}$, $i=1,2$, are face-centred cubic,
see Figure \ref{fig:zincblende}.
%
%
%In order to get a rigid tessellation, we will have to consider both nearest
%and next-to-nearest neighbours.
%%

Each atom of the sublattice $x\in\L\dc{_i}, i=1,2$, has four nearest neighbours
at distance $\frac{\sqrt6}4$, all belonging to the sublattice $\L\dc{_j} , j\neq i$.
Such bonds are not enough to provide a rigid tessellation of the space. Therefore we need
to take into account also the next-to-nearest neighbours.
By Definition \ref{def:NNN}, the next-to-nearest neighbours of $x$ in $\L\dc{}$
turn out to be its nearest neighbours as an element of $\L\dc{_i}$.
More precisely, each atom $x$ lies at the barycentre of a tetrahedron whose vertices are the nearest neighbours of $x$; the edges of such a tetrahedron correspond to next-to-nearest bonds.
Thus, when next-to-nearest neighbours are considered,
$\L\dc{}$ inherits some rigid structure from the (face-centred cubic) sublattices $\L\dc{_i}$, $i=1,2$.
\par
For a better understanding of the diamond cubic lattice,
we also refer to the simpler example of the planar honeycomb lattice,
which can be treated by the same methods as presented here. This two-dimensional example contains the main ideas for treating non-Bravais lattices
with next and next-to-nearest neighbours, see Figure \ref{fig:honey}.
\begin{figure}[p]
\centering
\includegraphics[width=.35\textwidth]{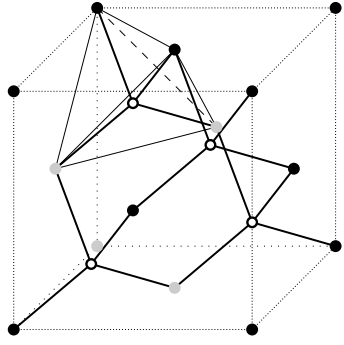}
\caption{Cubic cell in the diamond lattice $\L\dc{}$.
Atoms from the sublattice $\L\dc{}_1$ are represented in black/grey, while white atoms are from the sublattice $\L\dc{}_2$.
Nearest-neighbour bonds are displayed by solid thick lines.
Moreover, the picture shows a tetrahedron from the Delaunay pretriangulation of $\L\dc{}_1$:
its edges (solid and dashed thin lines) correspond to next-to-nearest neighbours in $\L\dc{}$.
A white atom lies at the barycentre of the tetrahedron, which is further divided
into four irregular tetrahedra by the bonds between the barycentre and each vertex.
}\label{fig:zincblende}
\end{figure}
\begin{figure}[p]
\centering
\includegraphics[width=.85\textwidth]{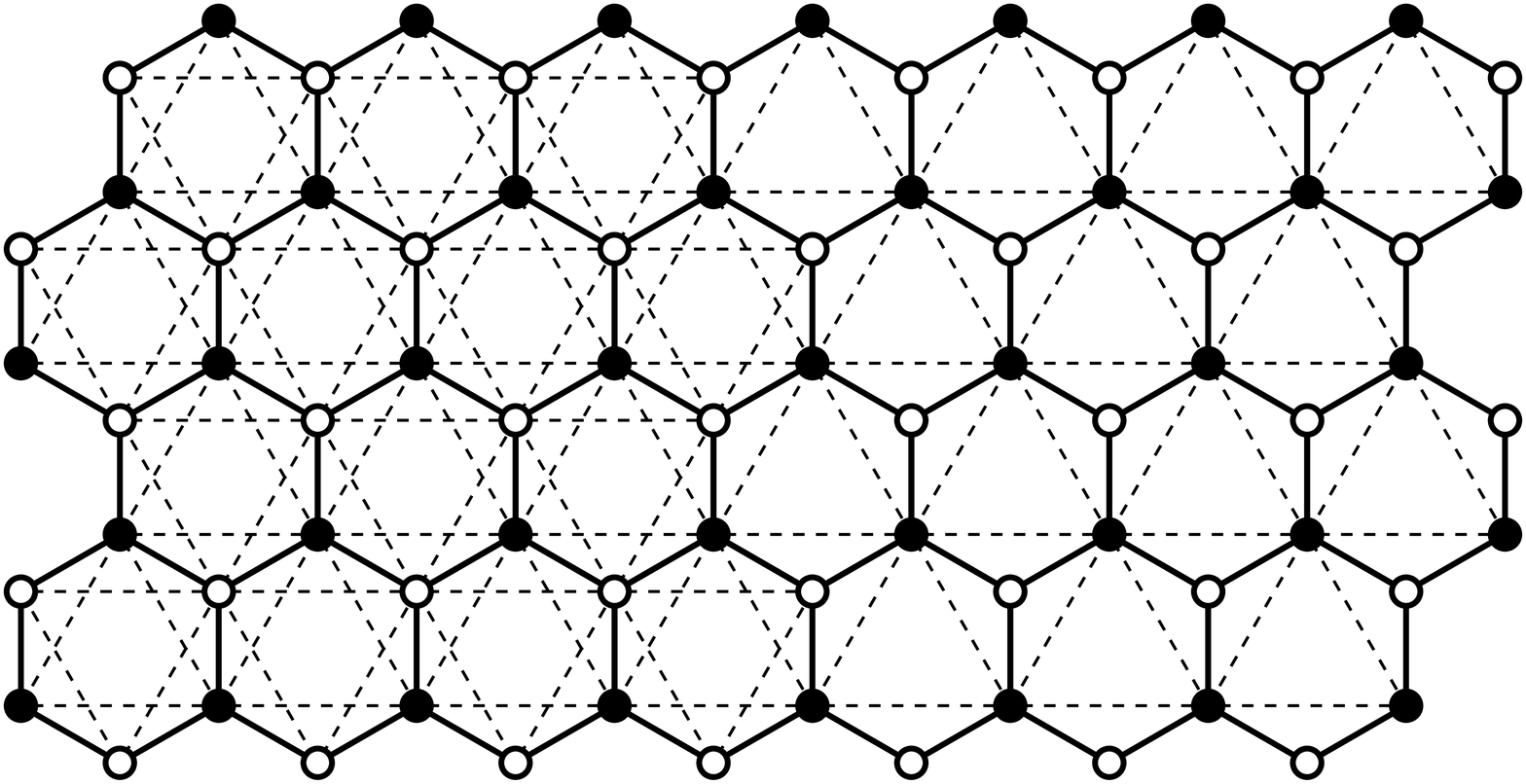}
\caption{Bonds and triangulation in a honeycomb lattice.
The lattice is given by
$\L\honey:=\L\honey_1 \cup \L\honey_2$, where
$\L\honey_i:=\{\u_i\honey +  \xi_1\wu\honey_1+\xi_2\wu\honey_2 \colon \ \xi_1,\xi_2 \in\Z\}$,
$\wu\honey_1:=(1,0)$, $\wu\honey_2:=(\frac12,\frac{\sqrt3}2)$,
$\u\honey_1:=(0,0)$, $\u\honey_2:=(0,\frac{\sqrt3}3)$.
This results into two interpenetrating sublattices $\L\honey_1$ and $\L\honey_2$,
both being hexagonal (i.e., equilateral triangular).
Atoms from $\L\honey_1$ and $\L\honey_2$ are displayed in different colors
in the picture, respectively in black and in white.
In the left part of the figure we indicate nearest neighbour (solid)
and next-to-nearest neighbour bonds (dashed lines).
The right part of the figure shows a possible triangulation,
that is the natural triangulation of $\L\honey_1$
enriched by considering the nearest-neighbour bonds between atoms $x\in\L\honey_1$ and $y\in\L\honey_2$.
This corresponds to ignoring the bonds between atoms of $\L\honey_2$, cf.\ Section~\ref{subsec:admiss}.
}
\label{fig:honey}
\end{figure}
%\clearpage
%
%
%
\section{Setting of the model}\label{sec:not}
In order to mathematically describe the three-dimen\-sional heterostructured nanowires we introduce four parameters $\eps$, $k$, $\lambda$ and $\rho$, next to the lattice structures discussed above.
\par
The parameter $\eps>0$ scales the equilibrium lattice distances and allows considering a passage from the discrete to the continuous setting by letting $\eps\to 0^{+}$. The parameter $k\in\N$, $k\ge1$, mimics the thickness of the nanowire. The shape of the nanowire in the discrete setting is a parallelepiped of length $2L$, $L>0$, and width and the height $k\eps$, see Section~\ref{sec:refconfintene} for details. In the continuum limit $\eps\to 0^{+}$, the length is conserved whereas the width and height tend to zero thus giving a dimension reduction of the system from three to one dimension. Still, the microscopic parameter $k$ has an impact on the continuum energy, which then allows investigating the limiting behaviour in dependence of the microscopic thickness $k$ of the wire.
\par
The parameters $\lambda$ and $\rho$ allow modeling the microscopic biphase structure of the nanowire. Here, $\lambda\in(0,1)$ denotes the ratio of the equilibrium distances in the deformed configuration of the material on the right hand side of the interface and of the material on the left hand side of the interface, see Section~\ref{subsec:setting} for details.
\par
The parameter $\rho\in(0,1]$ gives the ratio of the lattice distances of the two parts of the material in the reference configuration, where $\rho\in[\lambda,1]$ is the most interesting case. This allows treating different geometries of the nearest neighbours and in particular for dislocations.
The case of a defect-free body is modeled by $\rho=1$; the coordination number, i.e., the number of nearest neighbours of any internal atom, is constant in the lattice. If the crystal contains dislocations in the reference configuration, the coordination number is not constant. As we will show, this is the case for $\rho\neq 1$ and $k$ sufficiently large.
\subsection{Biphase lattices and rigid tessellations}
\label{subsec:setting}
Given $\rho\in(0,1]$ and vectors $\wu_1,\wu_2,\wu_3,$ $\u_1,\u_2\in\R^3$,
we define the  biphase atomistic lattice
\be
\label{def:biphase}
\LL_\rho:= \LL_1^-\cup\LL_\rho^+
\ee
by juxtaposing the two lattices
$\LL_1^-$ and $\LL_\rho^+$ given by
\begin{align*}
\L_1^- &:= \{\u_i + \xi_1\wu_1+\xi_2\wu_2+\xi_3\wu_3 \colon \ \xi_1,\xi_2,\xi_3\in\Z \,, \ i=1,2 \,, \ \xi_1<0 \} \,, \\
\L_\rho^+ &:= \{\rho\u_i + \xi_1\wu_1+\xi_2\wu_2+\xi_3\wu_3 \colon \ \xi_1,\xi_2,\xi_3\in\rho\Z \,, \ i=1,2 \,, \ \xi_1\ge0 \} \,.
%\LL_\rho&:= \LL_1^-\cup\LL_\rho^+\,.
\end{align*}
We will apply the above definitions to the crystals introduced in Section \ref{3d} and denote by
$\LL_\rho\fcc{}$, $\LL_\rho\hcp{}$, $\LL_\rho\bcc{}$, and $\LL_\rho\dc{}$
the lattices obtained by taking the vectors
$\{\wu_j\fcc{},\u_i\fcc{}\}_{\substack{j=1,2,3\\i=1,2\,\,\,}}$,
$\{\wu_j\hcp{},\u_i\hcp{}\}_{\substack{j=1,2,3\\i=1,2\,\,\,}}$,
$\{\wu_j\bcc{},\u_i\bcc{}\}_{\substack{j=1,2,3\\i=1,2\,\,\,}}$, and
$\{\wu_j\dc{},\u_i\dc{}\}_{\substack{j=1,2,3\\i=1,2\,\,\,}}$,
respectively, where $\u_i\fcc{}=u_i\bcc{}:=0$.
\par
In each of the four cases we find similar structures
for the planes at the interface between the lattices $\L_1^-$ and $\L_\rho^+$.
More precisely, for the face-centred cubic and the hexagonal close-packed,
the interfacial planes are two-dimensional equilateral triangular Bravais lattices, see Figure \ref{fig:fcc-int}.
In the body-centred cubic, the interfacial planes are triangular Bravais lattices, but not equilateral,
since the distance between nearest neighbours is not constant.
Finally, in the diamond cubic, whose properties are similar to the face-centred cubic,
we also find equilateral triangular planes composed by atoms of one of the sublattices.
(For a lower dimensional idea, see Figure \ref{fig:honey2}.)
\par
\begin{figure}[p]
\vspace{1cm} % ******************************************************************************************************
\centering
\subfloat[]{
\includegraphics[height=.28\textwidth]{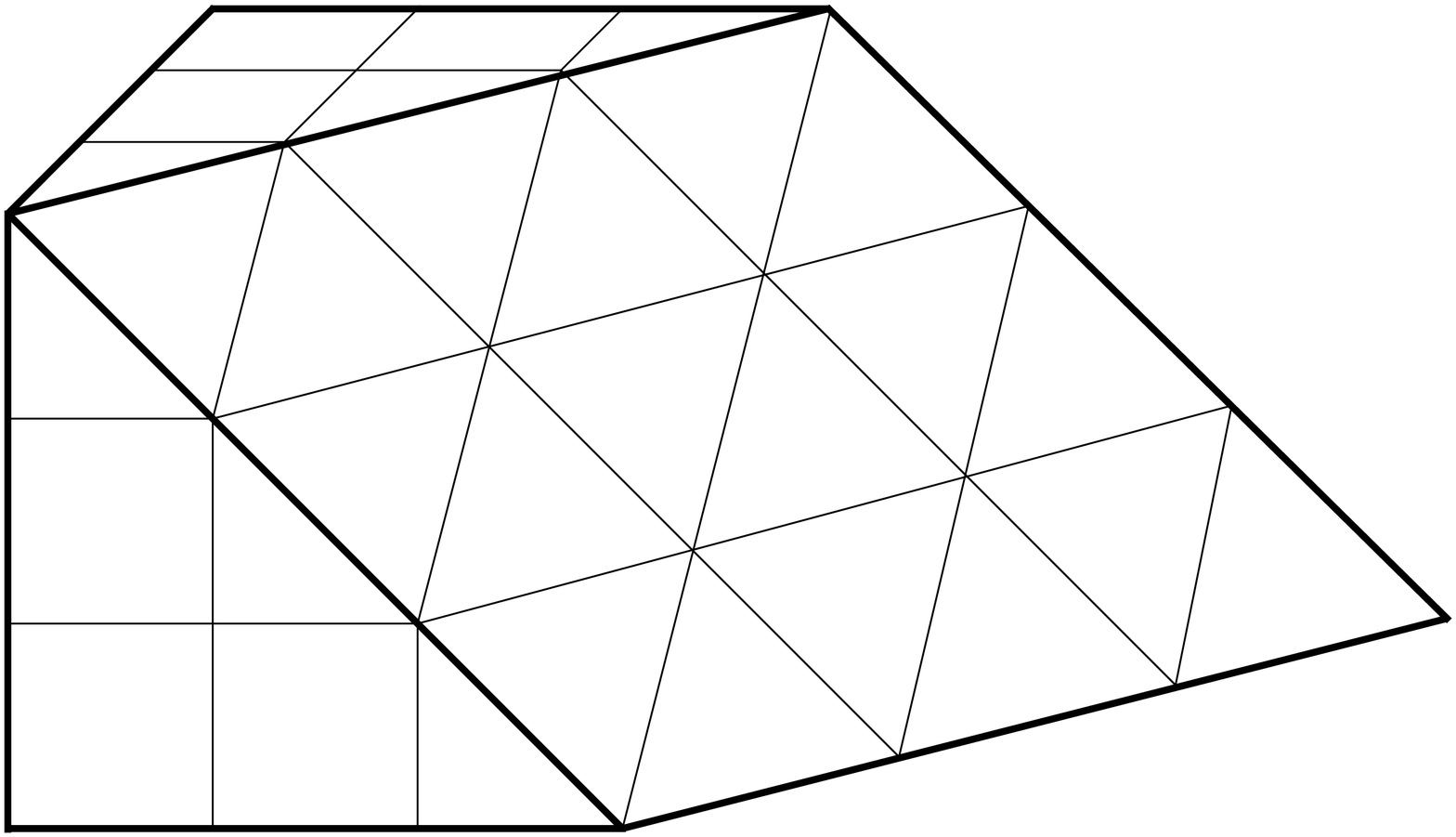}
\label{fig:interface21}
}
\hfill
\subfloat[]{
\includegraphics[height=.28\textwidth]{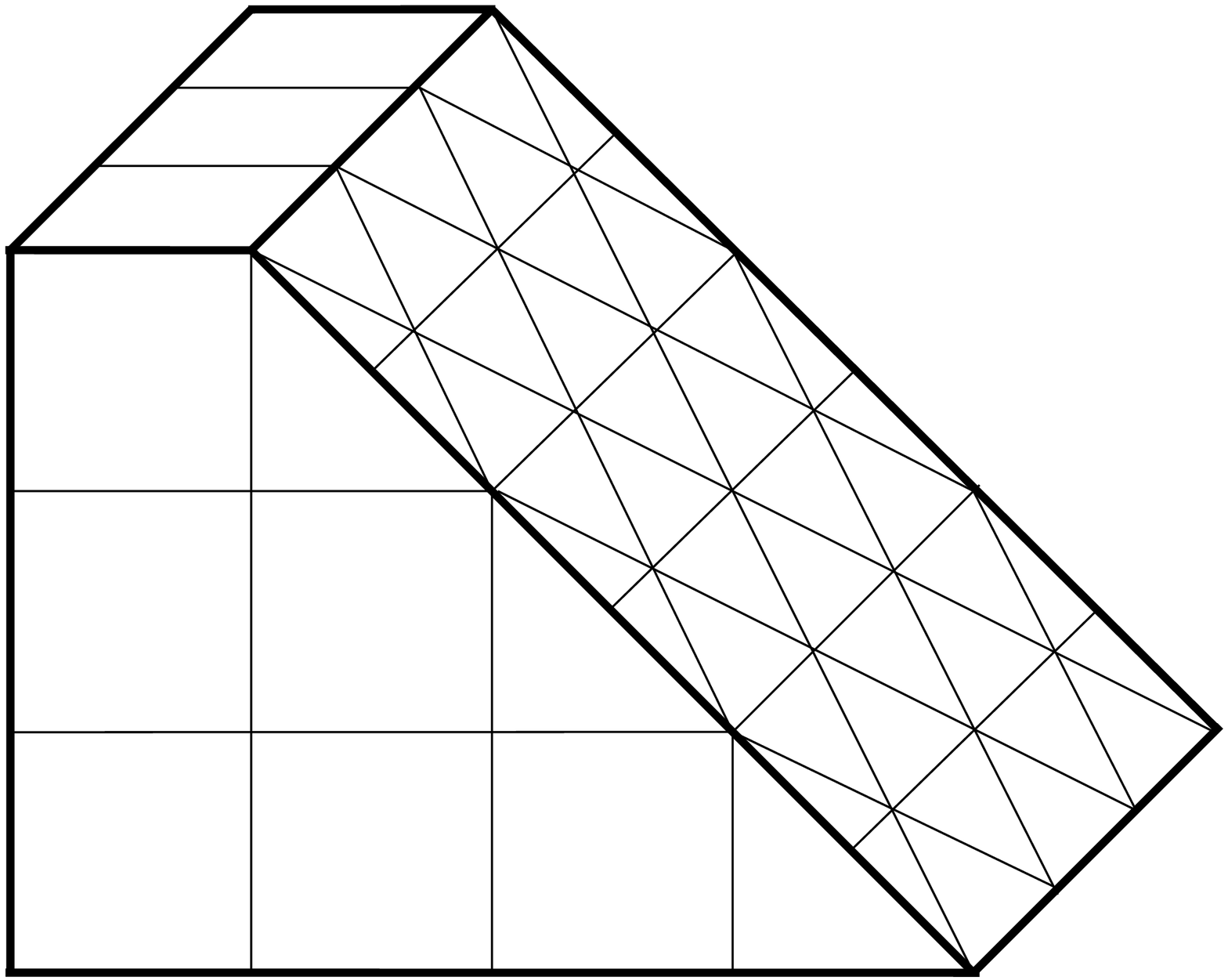}
\label{fig:interface31}
}
\caption{By cutting a cubic lattice along certain transverse planes,
one finds two-dimensional hexagonal Bravais lattices.
(a) face-centred; (b) body-centred.
}
\label{fig:fcc-int}
\end{figure}
\begin{figure}[p]
\vspace{1.5cm} % ******************************************************************************************************
\centering
\includegraphics[width=.85\textwidth]{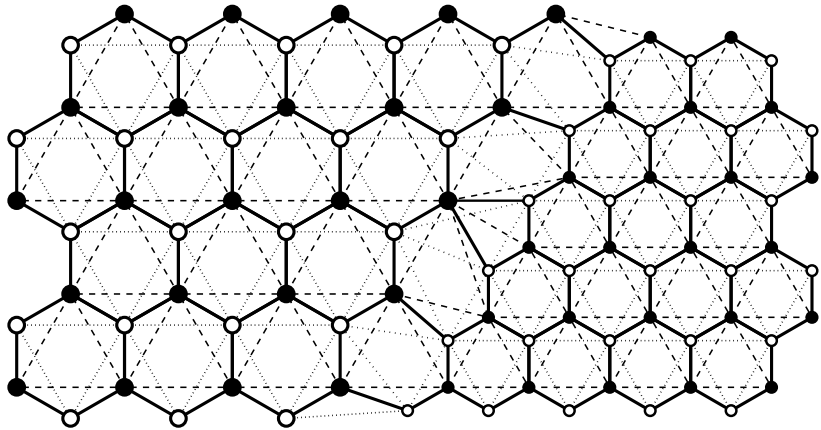}
\caption{Dislocations in a honeycomb-type lattice.
The bonds at the interface are chosen in the following way:
First one considers only black atoms and finds a Delaunay pretriangulation,
which is then refined to a triangulation (dashed lines);
the same is done for white atoms (dotted lines).
The dashed and dotted lines thus obtained give the bonds between next-to-nearest neighbours.
Finally, each white (resp.\ black) atom lying inside a triangle formed by three black (resp.\ white) atoms
is connected to the vertices of that triangle by nearest-neighbour bonds (solid lines).
}
\label{fig:honey2}
\end{figure}
\clearpage
Next we define the interfacial bonds in the case when $\rho\neq1$.
Following the idea already used in the regular parts of the lattices,
we consider the (unique) Delaunay pretriangulation $\TT'_\rho$ of $\LL_\rho$
(Definition \ref{def:Del-pre}).
This defines,
in the case of $\LL_\rho\fcc{}$, $\LL_\rho\hcp{}$, and $\LL_\rho\bcc{}$, a tessellation of the space into
rigid polyhedra away from the interface.
At the interface, the partition $\TT'_\rho$ may contain polyhedra with quadrilateral facets
(to see this, one should recall that the interfacial atoms lie on two parallel planes
consisting of two-dimensional triangular Bravais lattices, with parallel primitive vectors):
in such a case we refine  $\TT'_\rho$ further, in order to obtain rigid polyhedra.
More precisely, given a quadrilateral facet we introduce a further bond along a diagonal of the facet;
correspondingly, the region around the interface is subdivided into (irregular) tetrahedra and octahedra.
\par
Following this construction we define a partition of the space into rigid polyhedra
and call it the \emph{rigid Delaunay tessellation} associated to $\LL_\rho$, denoted by $\TT_\rho$.
The nearest neighbours are the extrema of the edges of the polyhedra of the subdivision.
Such procedure can be followed for $\LL_\rho\fcc{}$, $\LL_\rho\hcp{}$, and $\LL_\rho\bcc{}$.
% In the diamond cubic lattice $\LL_\rho\dc{}$ we follow a different construction.
Instead, in the diamond cubic lattice, applying Definition \ref{def:NN} may result in nearest-neighbour
bonds between interfacial atoms of the same sublattice % $\LL_{\rho}\dc{_i}$
(which should instead be next-to-nearest neighbours).
This would not be consistent with the structure defined away from the interface,
therefore we follow a different construction.
\par
Recall that $\LL_\rho\dc{}$ consists of two interpenetrating face-centred cubic lattices
$\LL_{\rho}\dc{_1}$ and $\LL_{\rho}\dc{_2}$, see \eqref{def:subdc} and \eqref{def:biphase}.
We introduce Delaunay pretriangulations for the sublattices $\LL_{\rho}\dc{_1}$ and
$\LL_{\rho}\dc{_2}$ individually,
which we further refine in order to obtain triangular facets
as before; we say that the vertices of the resulting edges are next-to-nearest neighbours
in $\LL_{\rho}\dc{}$.
The tessellation of $\LL_{\rho}\dc{_i}$ consists of (possibly irregular) tetrahedra and octahedra;
some of them may contain one atom $x$ of the other sublattice $\LL_{\rho}\dc{_j}$,
$j\neq i$. In this case,
we connect $x$ to the vertices of the surrounding polyhedron and say that each of those vertices
is a nearest neighbour for $x$.
When applied to the regular parts of the lattice, this construction is consistent with the notion
of nearest and next-to-nearest neighbours presented in the previous section.
(For a simpler idea about the resulting structure, we refer to Figure \ref{fig:honey2} in the case of a honeycomb-type lattice.)
\par
\subsection{Reference configurations and interaction energies} \label{sec:refconfintene}
We now pass to rescaled, bounded lattices.
Given $L>0$, $\eps\in(0,1]$, and $k\in\N$, we define
$$
\L_{\rho,\eps}(k):=(\eps\LL_\rho)\cap\overline\Om_{k\eps}\,,
$$
where
$$
\Om_{k\eps} := \{ \xi_1\wu_1+\xi_2\wu_2+\xi_3\wu_3\colon \xi_1\in(-L,L)\,,\ \xi_2,\xi_3\in(0, k\eps) \}
$$
and $\overline\Om_{k\eps}$ is the union of all (closed) polyhedra of $\TT_{\rho,\e}$ that intersect $\Om_{k\eps}$ on a nonempty set
(see \eqref{uffa} for the definition of $\TT_{\rho,\e}$). Notice that the lattices $\LL_{\rho,\e}\fcc{}$, $\LL_{\rho,\e}\hcp{}$, $\LL_{\rho,\e}\bcc{}$, and $\LL_{\rho,\e}\dc{}$ denote the corresponding rescaled, bounded lattices for the other crystal structures, with the lattice vectors chosen as above.
We set
\begin{align*}
\displaybreak[0]
\Om_{k\eps}^- &:= \{ \xi_1\wu_1+\xi_2\wu_2+\xi_3\wu_3\colon \xi_1\in(-L,0)\,,\ \xi_2,\xi_3\in(0, k\eps) \} \,,\\
\displaybreak[0]
\Om_{k\eps}^+ &:= \{ \xi_1\wu_1+\xi_2\wu_2+\xi_3\wu_3\colon \xi_1\in(0,L)\,,\ \xi_2,\xi_3\in(0, k\eps) \} \,,\\
\displaybreak[0]
\L_{1,\eps}^-(k) &:= \{\eps \u_i +  \xi_1\wu_1+\xi_2\wu_2+\xi_3\wu_3\in\L_{\rho,\eps}(k)\colon \xi_1<0\} \,,\\
\displaybreak[0]
\L_{\rho,\eps}^+(k) &:= \{\rho\eps\u_i + \xi_1\wu_1+\xi_2\wu_2+\xi_3\wu_3\in\L_{\rho,\eps}(k)\colon \xi_1\ge0\} \,.
\end{align*}
Two points $x,y\in \L_{\rho,\eps}$
are said to be nearest (resp., next-to-nearest) neighbours if $x/\eps$, $y/\eps$ fulfil the corresponding property in the lattice $\LL_\rho$.
This definition applies to each of the four cases presented above. Notice that we denote these lattices by $\L_{1,\eps}\fcc{^\pm}(k)$, $\L_{1,\eps}\hcp{^\pm}(k)$, $\L_{1,\eps}\bcc{^\pm}(k)$ and $\L_{1,\eps}\dc{^\pm}(k)$ as well as by $\L_{\rho,\eps}\fcc{^\pm}(k)$, $\L_{\rho,\eps}\hcp{^\pm}(k)$, $\L_{\rho,\eps}\bcc{^\pm}(k)$ and $\L_{\rho,\eps}\dc{^\pm}(k)$, respectively.
\par
We have introduced so far the bonds that enter the definition of the energy, which we generally
denote by  $\E_{\eps}^{\lambda}$.
%Let now $\lambda\in(0,1)$ be fixed.
Next we specialise $\E_{\eps}^{\lambda}$ for each of the four lattices introduced above.
In the cases of the face-centred cubic and of the hexagonal close-packed,
the total interaction energy is defined respectively by
\be\label{en:fcc}
\begin{split}
\E_{\eps}\fcc{,\,\lambda}(u_\eps,\rho,k) := &\
\tfrac{1}{2}\!\!\! \sum_{\substack{\NN{x}{y}\\x\in \L_{1,\eps}\fcc{^-}(k) \\y\in\L_{\rho,\eps}\fcc{}(k)}}
\left(\Big|\frac{u_\eps(x)-u_\eps(y)}{\e}\Big|-1\right)^2
\\ &
+
\tfrac{1}{2}\!\!\! \sum_{\substack{\NN{x}{y}\\x\in \L_{\rho,\eps}\fcc{^+}(k)\\y\in\L_{\rho,\eps}\fcc{}(k)}}
\left(\Big|\frac{u_\eps(x)-u_\eps(y)}{\e}\Big|-\lambda\right)^2
\end{split}
\ee
for every deformation $u_\eps\colon\LL\fcc{}_{\rho,\eps}(k)\to\R^3$
and by
\be\label{en:hcp}
\begin{split}
\E_{\eps}\hcp{,\,\lambda}(u_\eps,\rho,k) := &\
\tfrac{1}{2}\!\!\! \sum_{\substack{\NN{x}{y}\\x\in \L_{1,\eps}\hcp{^-}(k) \\y\in\L_{\rho,\eps}\hcp{}(k)}}
\left(\Big|\frac{u_\eps(x)-u_\eps(y)}{\e}\Big|-1\right)^2
\\ &
+
\tfrac{1}{2}\!\!\! \sum_{\substack{\NN{x}{y}\\x\in \L_{\rho,\eps}\hcp{^+}(k)\\y\in\L_{\rho,\eps}\hcp{}(k)}}
\left(\Big|\frac{u_\eps(x)-u_\eps(y)}{\e}\Big|-\lambda\right)^2
\end{split}
\ee
for every deformation $u_\eps\colon\LL\hcp{}_{\rho,\eps}(k)\to\R^3$.
\par
For the body-centred cubic, we need to use an anisotropic energy,
because of the different length of the bonds between nearest neighbours
in the reference configuration.
For every deformation $u_\eps\colon\LL\bcc{}_{\rho,\eps}(k)\to\R^3$ we define
\be\label{bcc-energy}
\begin{split}
\E_{\eps}\bcc{,\,\lambda}(u_\eps,\rho,k) := &\
\tfrac{1}{2}\!\!\! \sum_{\substack{\NN{x}{y}\\x\in \L_{1,\eps}\bcc{^-}(k) \\y\in\L_{\rho,\eps}\bcc{}(k)}}
\left(\Big|\frac{u_\eps(x)-u_\eps(y)}{\e}\Big|-\var^x_y\right)^2
\\ &
+
\tfrac{1}{2}\!\!\! \sum_{\substack{\NN{x}{y}\\x\in \L_{\rho,\eps}\bcc{^+}(k)\\y\in\L_{\rho,\eps}\bcc{}(k)}}
\left(\Big|\frac{u_\eps(x)-u_\eps(y)}{\e}\Big|-\lambda\,\var^x_y\right)^2 \,,
\end{split}
\ee
where $\var^x_y:=\var(\frac{x-y}{\mod{x-y}})$ and
$\var\colon \mathbb{S}^2\to(0,+\infty)$ is a smooth function such that
\begin{alignat*}{2}
\var(\wu)&=\sqrt2 \quad&&\text{if } \wu\in\{(\pm1,0,0),(0,\pm1,0),(0,0,\pm1)\} \,, \\
\var(\wu)&=\tfrac{\sqrt6}{2} \quad&&\text{if } \sqrt3\,\wu\in\{(\pm1,1,1), (\pm1,-1,1), (\pm1,1,-1), (\pm1,-1,-1)\}.
\end{alignat*}
\par
Finally, recall that the diamond cubic lattice consists of two interpenetrating face-centred  cubic lattices. Therefore we set for a deformation $u_\eps\colon\LL\dc{}_{\rho,\eps}(k)\to\R^3$
\bes
\E_{\eps}\dc{,\,\lambda}(u_\eps,\rho,k) :=
c_1 \E\dc{_1}_{{\rm NNN}} + c_2 \E\dc{_2}_{{\rm NNN}} + \E\dc{}_{{\rm NN}} \,,
\ees
where the first two summands account for next-to-nearest neighbour interactions and are defined as in \eqref{en:fcc}, namely
\bes
 \E\dc{_i}_{{\rm NNN}} :=
\tfrac{1}{2}\!\!\! \sum_{\substack{\NNN{x}{y}\\x\in \L_{1,\eps}\dc{_i^-}(k) \\y\in\L_{\rho,\eps}\dc{}(k)}}
\left(\Big|\frac{u_\eps(x)-u_\eps(y)}{\e}\Big|-1\right)^2 +
\tfrac{1}{2}\!\!\! \sum_{\substack{\NNN{x}{y}\\x\in \L_{\rho,\eps}\dc{_i^+}(k)\\y\in\L_{\rho,\eps}\dc{}(k)}}
\left(\Big|\frac{u_\eps(x)-u_\eps(y)}{\e}\Big|-\lambda\right)^2 \,,
\ees
while the last term is
\bes
\E\dc{}_{{\rm NN}} :=
\tfrac{1}{2}\!\!\! \sum_{\substack{\NN{x}{y}\\x\in \L_{1,\eps}\dc{^-}(k) \\y\in\L_{\rho,\eps}\dc{}(k)}}
\left(\Big|\frac{u_\eps(x)-u_\eps(y)}{\e}\Big|-\tfrac{\sqrt6}{4}\right)^2 +
\tfrac{1}{2}\!\!\! \sum_{\substack{\NN{x}{y}\\x\in \L_{\rho,\eps}\dc{^+}(k)\\y\in\L_{\rho,\eps}\dc{}(k)}}
\left(\Big|\frac{u_\eps(x)-u_\eps(y)}{\e}\Big|-\lambda\tfrac{\sqrt6}{4}\right)^2 \,.
\ees
The choice of the constants $c_1,c_2>0$ determines how strong the interactions between
atoms of the same sublattice $\L\dc{_i}$ are.
\par
\subsection{Admissible configurations}
\label{subsec:admiss}
In order to define the admissible deformations, we introduce piecewise affine functions.
To this end, we need to refine $\TT_\rho$ to a proper triangulation.
However, we do not change the definition of the nearest neighbours, i.e., we do not introduce new interactions in the energy.
\begin{remark}
For the reader's convenience, we summarise here the different tessellations of the space associated to a biphase discrete lattice $\LL_\rho$,
adopted in our setting.
\begin{itemize}
\item We have started from the (unique) \emph{Delaunay pretriangulation} $\TT'_\rho$ (Definition \ref{def:Del-pre}),
which may contain non-rigid polyhedra at the interface.
\item We have refined $\TT'_\rho$, obtaining a \emph{rigid Delaunay tessellation} $\TT_\rho$,
a partition of the space into (possibly irregular) tetrahedra and octahedra.
Such a tessellation is not unique, indeed we have chosen a diagonal
for each quadrilateral facet of polyhedra of $\TT'_\rho$.
The corresponding bonds enter the definition of the interaction energy.
\item In order to work with piecewise affine functions, in this section
we further refine $\TT_\rho$ to get three possible \emph{triangulations}
(i.e., subdivisions of the space into tetrahedra only), denoted by
$\TT_\rho^{(1)}$, $\TT_\rho^{(2)}$, and $\TT_\rho^{(3)}$, respectively.
\end{itemize}
The above construction is used to work in the case of
$\LL_\rho\fcc{}$, $\LL_\rho\hcp{}$, and $\LL_\rho\bcc{}$.
For  $\LL_\rho\dc{}$, the definition of $\TT_\rho$ is different,
as made precise in Sections \ref{subsec:setting} and \ref{subsec:admiss}.
\end{remark}
In the case of  $\LL_\rho\fcc{}$, $\LL_\rho\hcp{}$, and $\LL_\rho\bcc{}$,
given a (possibly irregular) octahedron of $\TT_\rho$, we divide it into four irregular tetrahedra
by cutting it along one of the three diagonals.
We choose the diagonal starting from the vertex with the largest $x_1$-coordinate;
if two or three vertices have the same largest $x_1$-coordinate, we take among them the point with largest $x_2$-coordinate;
if two of such vertices have also the same largest $x_2$-coordinate, we take the one with the largest $x_3$-coordinate.
By repeating the process on every octahedron of $\TT_\rho$, we obtain a triangulation that we denote by $\TT_\rho^{(1)}$.
Other two triangulations $\TT_\rho^{(2)}$ and $\TT_\rho^{(3)}$ are obtained by repeating the same procedure,
but with different ordering of the indices, namely $x_2,x_3,x_1$ and $x_3,x_1,x_2$ respectively.
\par
In the case of the diamond-cubic lattice, we define a triangulation as follows:
we consider the Delaunay pretriangulation of $\L\dc{_1}$, which is rigid.
As already observed, some of the tetrahedra of the latter pretriangulation contain
an atom of $\L\dc{_2}$
at the barycentre (more precisely, every other tetrahedron has this property,
see Figure \ref{fig:zincblende}).
Such tetrahedra are further subdivided by connecting
the barycentre to the vertices.
In other words, we define a tessellation into tetrahedra and octahedra by considering
the (nearest neighbour) interactions between atoms
$x\in\L\dc{_1}$ and $y\in\L\dc{_2}$,
as well as the interactions between atoms of $\L\dc{_1}$
(nearest neighbour if restricted to $\L\dc{_1}$,
next-to-nearest neighbour if viewed in the whole $\L\dc{}$),
and ignoring the interactions between atoms of $\L\dc{_2}$.
We apply the same rule to the biphase lattice $\LL_\rho\dc{}$ and
further subdivide the resulting octahedra as done for  $\LL_\rho\fcc{}$, $\LL_\rho\hcp{}$,
and $\LL_\rho\bcc{}$,
obtaining three possible triangulations.
For a better
understanding we illustrate the tessellation thus defined in the simpler case
of the honeycomb lattice in Figure \ref{fig:honey}.
\par
Given a function $u\colon \L_\rho\to\R^3$, we denote by $u^{(1)}$, $u^{(2)}$, and $u^{(3)}$ its piecewise affine interpolations
with respect to the triangulations $\TT_\rho^{(1)}$, $\TT_\rho^{(2)}$, and $\TT_\rho^{(3)}$, respectively.
Analogous definitions and notations hold for the rescaled bounded lattices of the type $\L_{\rho,\eps}(k)$.
More precisely, define
\be\label{uffa}
\TT_{\rho,\eps}:=\{\eps T\colon T\in\TT_\rho\} \quad \text{and} \quad \TT_{\rho,\eps}^{(i)}:=\{\eps T\colon T\in\TT_\rho^{(i)}\}
\ee
for $i=1,2,3$.
The set of admissible deformations is
\be\label{ad-3d}
\begin{split}
\A_{\rho,\eps}(\Om_{k\eps}):= \big\{ u_\eps\in C^0(\overline\Om_{k\eps};\R^3) \colon & u_\eps \ \text{piecewise affine,}\\
& \D u_\eps \ \text{constant on}\ \Om_{k\eps}\cap T\quad \forall\, T\in\TT_{\rho,\eps}^{(1)}\,, \\
& \det \D u_\eps>0 \ \text{a.e.\ in}\ \Om_{k\eps}\,, \\
& u_\eps(\mathcal P) \ \text{is convex}\ \forall\, \mathcal P\in\TT_{\rho,\eps} \big\} \,.
\end{split}
\ee
The restriction of $u_\eps\in\A_{\rho,\eps}(\Om_{k\eps})$ to $\L_{\rho,\eps}(k)$ is still denoted by $u_\eps$.
We will see that the limiting functional is independent of the choice of the triangulation $\TT_{\rho,\eps}^{(1)}$ in \eqref{ad-3d}, cf.\ Remark~\ref{rmk5}.
\par
\begin{remark}\label{rmk:3d}
The assumption of convexity on the images of the octahedra of $\TT_{\rho,\eps}$ is needed to enforce rigidity:
without such an assumption an octahedron could be compressed without paying any energy.
On the other hand,  the notion of non-interpenetration used in \eqref{ad-3d} is independent of the choice of the triangulation $\TT_{\rho,\eps}^{(1)}$
provided the image of each octahedron is assumed to be convex, as clarified by Lemma \ref{lemma:convex}.
\end{remark}
It will be convenient to introduce
\bes
 \Om_{k,\infty}:=\,
 \{ \xi_1\wu_1+\xi_2\wu_2+\xi_3\wu_3\colon \xi_1\in(-\infty,+\infty)\,,\ \xi_2,\xi_3\in(0, k) \}
\ees
and to denote by $\overline\Om_{k,\infty}$ the union of all (closed) polyhedra of $\TT_{\rho}$ that have a nonempty intersection with $\Om_{k,\infty}$.
% \L_{\rho,\infty}(k):=&\,\LL_\rho\cap\overline\Om_{k,\infty} \,,\\
% \L_{1,\infty}^-(k) :=&\, \{\xi_1\w_1+\xi_2\w_2\in\L_{\rho,\infty}(k)\colon \xi_1<0\} \,,\\
% \L_{\rho,\infty}^+(k) :=&\, \{\xi_1\w_1+\xi_2\w_2\in\L_{\rho,\infty}(k)\colon \xi_1\ge0\} \,.
We define the set of admissible deformations on the rescaled infinite domain
as
\bes
\begin{split}
\A_{\rho,\infty}(\Om_{k,\infty}):= \big\{ u\in C^0(\overline\Om_{k,\infty};\R^3) \colon & u \ \text{piecewise affine,}\\
& \D u \ \text{constant on}\ \Om_{k,\infty}\cap T\ \forall\, T\in\TT_{\rho}\,, \\
& \det \D u>0 \ \text{a.e.\ in}\ \Om_{k,\infty}\,, \\
&u (\mathcal P) \ \text{is convex}\ \forall\, \mathcal P\in\TT_{\rho} \big\} \,.
\end{split}
\ees
All definitions apply to each of the four cases presented above.
Correspondingly, we define the energy on the rescaled infinite domain and denote it by
$\E_{\infty}^\lambda$.
Specifically, given a discrete deformation $v$ of the face-centred cubic lattice, $\E_{\infty}^\lambda$ is defined by
\be\label{energia-infinita}
\E_{\infty}\fcc{,\,\lambda}(v,\rho,k) :=
\tfrac{1}{2}\!\!\! \sum_{\substack{\NN{x}{y}\\x\in \L_{1}\fcc{^-}(k) \\y\in\L_{\rho}\fcc{}(k)\\x,y\in\overline\Om_{k,\infty}}}
\left(\big| v(x)-v(y) \big|-1\right)^2 +
\tfrac{1}{2}\!\!\! \sum_{\substack{\NN{x}{y}\\x\in \L_{\rho}\fcc{^+}(k)\\y\in\L_{\rho}\fcc{}(k)\\x,y\in\overline\Om_{k,\infty}}}
\left(\big| v(x)-v(y) \big|-\lambda\right)^2 \,.
\ee
Analogous definitions hold for $\LL_\rho\hcp{}$, $\LL_\rho\bcc{}$ and
 $\LL_\rho\dc{}$.
\section{Discrete rigidity in dimension three}\label{sec:rigidity}
A key tool in the analysis developed in \cite{LaPaSc2014} for two-dimensional heterogeneous nanowires as well as in the analysis of the three-dimensional setting is the following rigidity estimate.
%of Friesecke, James and M\"uller \cite{fjm}.
%
\begin{theorem}\label{thm-rigidity}
\cite[Theorem 3.1]{fjm}
Let $N\geq 2$, and let $1< p < +\infty$.
Suppose that $U\subset\R^{N}$ is a bounded Lipschitz domain.
Then there exists a constant $C=C(U)$
such that for each $u\in W^{1,p}(U;\R^{N})$
there exists a constant matrix $R\in SO(N)$ such that
\begin{equation}\label{rigidity}
\|\D  u-R\|_{L^{p}(U;\Mnn)} \leq C(U)
\|\dist(\D  u,SO(N))\|_{L^{p}(U)}\,.
\end{equation}
The constant $C(U)$ is invariant under dilation and translation of the domain.
\end{theorem}
In order to employ  the above result, we need the discrete rigidity estimates of Lemmas \ref{lemma:tetra} and \ref{lemma:ottaedro}, which state that the energy of a lattice cell is bounded from below
by the distance of the deformation gradient from the set of rotations.
Similar rigidity estimates are used in \cite{BSV,FlaThe,Schm06,Th06}.
\par
We use the following notation for the vectors determined by the edges of the regular tetrahedron $\mathcal S$ of edge length one:
$\w_1:=(1,0,0)$, $\w_2:=(\frac{1}{2},\frac{\sqrt{3}}{2},0)$, $\w_3:=\w_2-\w_1$,
$\w_4:=(\frac{1}{2},\frac{\sqrt{3}}{6},\frac{\sqrt{6}}{3})$,
$\w_5:=\w_4-\w_2$, and $\w_6:=\w_4-\w_1$, cf.\ Figure~\ref{tetraedro}.
\par
\begin{lemma}\label{lemma:tetra}
There exists $C>0$ such that
\be\label{rig-tetraedro}
\dist^2(F,SO(3)) \leq  C \sum_{i=1}^6( |F\w_i|-1)^2
\ee
for every  $F\in GL^+(3)$.
\end{lemma}
\begin{proof}
%%%%%%%%%%%%%%%%%%%%%%%%%%%%%%%%%%%%%%%%%%%%%%%%%
\begin{figure}
\centering
\psfrag{w1}{$\w_1$}
\psfrag{w2}{$\w_2$}
\psfrag{w3}{$\w_3$}
\psfrag{w4}{$\w_4$}
\psfrag{fw1}{$F\w_1$}
\psfrag{fw2}{$F\w_2$}
\psfrag{fw4}{$F\w_4$}
\includegraphics[width=.5\textwidth]{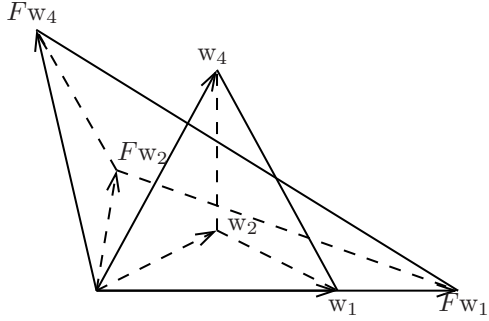}
\caption{The tetrahedron $\mathcal S$ and its image $F(\mathcal S)$.}
\label{tetraedro}
\end{figure}
Set $\delta_i:= |F \w_i| -1$ and $\delta:=(\delta_1,\dots,\delta_6)$, then $\sum_{i=1}^6( |F\w_i|-1)^2 = \sum_{i=1}^6 \delta_i^2=|\delta|^2$.
Without loss of generality we may assume that
$$
F \w_1 =(1+\delta_1)\w_1\,, \quad
F\w_2 \in {\rm{span}}\{\w_1,\w_2\}\,,\quad
F\w_2 \cdot e_2 >0 \,,
$$
as in Figure \ref{tetraedro}. Notice that the above assumptions imply $F\w_4 \cdot e_3 >0$.
We have
\begin{align}\label{dist-est-2}
\nonumber
\dist^2(F,SO(3)) \leq |F-I|^2 & \leq C \big( | (F-I) \w_1 |^2 +   |(F-I)\w_2 |^2   +  |(F-I)\w_4 |^2  \big)  \\
&=C \big(\delta_1^2  + | (F-I)\w_2 |^2 +  |(F-I)\w_4 |^2   \big) \,.
\end{align}
By a simple geometric argument one finds
\be\label{dist-estimate}
|(F-I)\w_2 |^2 = 1 + (1+\delta_2)^2 -2(1+\delta_2)\cos\big(\theta_{12} - \tfrac{\pi}{3}\big) \,,
\ee
where $\theta_{12}$ is the angle (measured anticlockwise) between $\w_1$ and $F\w_2$,
which is determined by
\be\label{dist-estimate2}
\cos\theta_{12} =
\frac{ (1+\delta_1)^2  +  (1+\delta_2)^2 - (1+\delta_3)^2 }{ 2 (1+\delta_1)(1+\delta_2)}
\quad \text{and} \quad \sin\theta_{12} >0 \,,
\ee
cf.\ \cite[Proof of Lemma~2.2]{LaPaSc2014}.
Notice that the condition $\sin\theta_{12}>0$ follows from the assumptions $F\in GL^+(3)$
and $F\w_2 \cdot e_2 >0 $.
\par
Denote by $\theta_{ij}$ the acute angle formed by $F\w_i$ and $F\w_j$ and by $\eta_{44}$
that between $F\w_4$ and $\w_4$.
Since
\be\label{formula:F-I}
 |(F-I)\w_4 |^2 = 1+  (1+\delta_4)^2   - 2 (1+\delta_4)\cos\eta_{44}\,,
 \ee
 in order to express
the right hand side of \eqref{dist-est-2} in terms of the $\delta_i$'s, we need to specialize
$\cos\eta_{44}$ in terms of the $\delta_i$'s.
Set
$$
\frac{F\w_4}{|F\w_4|}: =(a_1,a_2,a_3) \,,
$$
and remark that, by assumption, $a_3 >0$.
Thus
\begin{equation}\label{tildecos44}
\cos\eta_{44}= \frac{F\w_4}{|F\w_4|} \cdot \w_4 =
\tfrac{1}{2}a_1 + \tfrac{\sqrt{3}}{6}a_2 +  \tfrac{\sqrt{6}}{3}\sqrt{1- a_1^2 -a_2^2} \,.
\end{equation}
On the other hand $a_1$ and $a_2$ are computed by solving
\begin{align}\label{coordinate}
& a_1 = \frac{F\w_4}{|F\w_4|} \cdot \w_1 = \cos\theta_{14} \,, \\
&
a_1  \cos\theta_{12} + a_2  \sin\theta_{12} =
\frac{F\w_4}{|F\w_4|} \cdot \frac{F\w_2}{|F\w_2|} = \cos\theta_{24}\,,
\end{align}
where
\be
\cos\theta_{14} =
\frac{ (1+\delta_1)^2  +  (1+\delta_4)^2 - (1+\delta_6)^2 }{ 2 (1+\delta_1)(1+\delta_4)}\,,
\ee
\begin{equation}\label{cos24}
\cos\theta_{24} =
\frac{ (1+\delta_2)^2  +  (1+\delta_4)^2 - (1+\delta_5)^2 }{ 2 (1+\delta_2)(1+\delta_4)}\,.
\end{equation}
Taking into account \eqref{dist-estimate}--\eqref{cos24}, one can express the right hand side of
\eqref{dist-est-2} as a function $f$ of  $\delta$ and see that $f(0)=0$ and $\nabla f(0)=0$,
which implies $f(\delta)\leq C|\delta|^2$ for $|\delta|$ sufficiently small.
For larger $|\delta|$ the inequality readily follows from \eqref{dist-est-2}--\eqref{formula:F-I}.
\end{proof}
%
%
%
%%%%%%%%%%%%%%%%%%%%%%%%%%%%%%%%%%%%%%
%From \eqref{dist-est-2} and \eqref{dist-estimate} we deduce that
%%
%\begin{equation}\label{big-delta}
%\dist^2(F,SO(2)) \leq C(1+ \delta_1^2 + \delta_2^2) \,.
%\end{equation}
%%
%On the other hand, by computing the second order Taylor expansion of \eqref{dist-estimate} about the point $\delta=(0,0,0)$ and taking into account \eqref{dist-estimate2} we see that
%%
%\begin{equation}\label{small-delta}
%\dist^2(F,SO(2)) \leq C |\delta|^2 + o(|\delta|^2) \,,
%\end{equation}
%%
%all first derivatives being zero at $(0,0,0)$.
%Then \eqref{bound-from-below} readily follows from \eqref{small-delta} for $\E(F)$ small,
%from \eqref{big-delta} for values of $\E(F)$ larger than one,
%and by boundedness in the intermediate case.
%%%%%%%%%%%%%%%%%%%%%%%%%%%%%%%%%%%%%%
%
We will consider the octahedron $\mathcal O$ generated by the points
$P_1:=(0,0,0)$, $P_2:=(1,0,0)$, $P_3:=(0,1,0)$, $P_4:=(1,1,0)$, $P_5:=(\frac12,\frac12,\frac{\sqrt{2}}{2})$,
and $P_6:=(\frac12,\frac12,-\frac{\sqrt{2}}{2})$, see Figure \ref{ottaedro}.
\begin{figure}
\centering
\psfrag{1}{$P_1$}
\psfrag{2}{$P_2$}
\psfrag{3}{$P_3$}
\psfrag{4}{$P_4$}
\psfrag{5}{$P_5$}
\psfrag{6}{$P_6$}
\includegraphics[width=.35\textwidth]{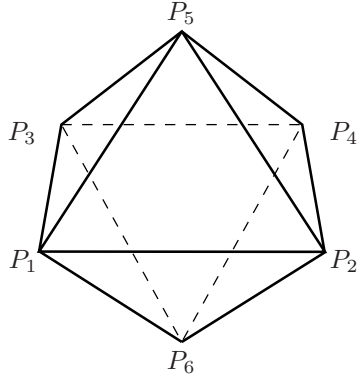}
\caption{The octahedron $\mathcal O$.}
\label{ottaedro}
\end{figure}
We call $\TT^{(1)}$ the triangulation determined by cutting $\mathcal O$ along the diagonal $P_1P_4$, further $\TT^{(2)}$ denotes the triangulation corresponding to $P_2P_3$, and $\TT^{(3)}$ the one corresponding to $P_5P_6$.
\par
Given a deformation $u$ of the six vertices of $\mathcal O$,
$u^{(i)}$ denotes the piecewise affine extension of $u$ corresponding to
the triangulation $\TT^{(i)}$, $i=1,2,3$.
In the next lemma, $\mathcal Q$ denotes the interior of the (bounded) polyhedron determined by the images of the facets of $\mathcal O$ through any of the piecewise affine extensions $u^{(i)}$.
(Notice that the images of the facets do not depend on the chosen extension.)
\begin{lemma}\label{lemma:convex}
One has that $\det \D u^{(1)}>0$ a.e.\ in $\mathcal O$ and $\mathcal Q$ is convex if and only if $\det \D u^{(i)}>0$ for every $i=1,2,3$.
\end{lemma}
\begin{proof}
Assume that $\det \D u^{(1)} > 0$ a.e. This implies that $\mathcal Q$ is connected, the diagonal $u(P_1)u(P_4)$ is contained in $\mathcal Q$, and the outer normal vectors to the facets of $\mathcal Q$ point towards the outside of $\mathcal Q$.
Consider now the tetrahedra of the triangulation $\TT^{(2)}$: since the normals to the facets of $\mathcal Q$ point towards the outside of $\mathcal Q$,
it turns out that $\det \D u^{(2)}> 0$ a.e. if and only if the diagonal $u(P_2)u(P_3)$ is contained in $\mathcal Q$. The same holds for $\TT^{(3)}$ using the corresponding diagonal.
On the other hand, an octahedron is convex if and only if all the three diagonals are contained in the inner part of the octahedron itself.
\end{proof}
%
% \begin{proof}
% \arancio{Let $\TT^{(1)}(\mathcal{P}):=\{T\in\TT_{\rho,\eps}^{(1)}\colon T\subset\mathcal{P}\}$.}
% Since $\TT^{(1)}(\mathcal{P})$ is positively oriented, if follows that $u_\eps(\partial\mathcal O)$ divides the space
% in two connected components, the bounded connected component coincides with $u_\eps(\mathcal O)$,
% and the images of all tetrahedra of $\TT_{\rho,\eps}^{(1)}$ are contained in $u_\eps(\mathcal O)$.
% Next we remark that a diagonal is contained in $u_\eps(\mathcal O)$ if and only if the corresponding triangulation is positively oriented.
% On the other hand, an octahedron is convex if and only if all the three diagonals are contained in the inner part of the octahedron itself.
% \end{proof}
%
%
The octahedron satisfies an estimate corresponding to the one of Lemma \ref{lemma:tetra}.
%by subdividing it into four tetrahedra and using Remark \ref{rmk:3d} and Lemma \ref{lemma:tetra}.
%
\begin{lemma}\label{lemma:ottaedro}
There exists $C>0$ such that
\begin{equation}\label{rig-ottaedro}
\dist^2(\D u,SO(3)) \leq  C \!\!\sum_{\NN{P_i}{P_j}}( |\D u(P_iP_j)|-1)^2 \quad
\text{ a.e.\ in }\mathcal O\,,
\end{equation}
for every $u\in C^0(\mathcal O;\R^3)$ such that $u$ is piecewise affine
with respect to the triangulation determined by cutting $\mathcal O$ along the
diagonal $P_1P_4$,
%vector $(1,1,0)$,
$\det\D u>0$ a.e.\ in $\mathcal O$,
and $u(\mathcal O)$ is convex.
\end{lemma}
\begin{proof}
Let  $\chi_i$, $i=1,\dots,4$, be the characteristic functions of the
four tetrahedra $T_1:=P_1P_2P_4P_5$, $T_2:=P_1P_2P_4P_6$,
$T_3:=P_1P_3P_4P_5$, $T_4:=P_1P_3P_4P_6$,
respectively.
Since $\nabla u = \sum_{i=1}^4 \chi_i F_i$ for some $F_i\in GL^+(3)$, it suffices to prove \eqref{rig-ottaedro} in each tetrahedron.
Notice that
$P_1$ and $P_4$ are not nearest neighbours
and therefore we cannot directly apply Lemma~\ref{lemma:tetra}.
On the other hand, the length of  $u(P_1P_4)$, which is a common edge of the four  deformed
tetrahedra, can be expressed as a function of all the edges  of $u(\mathcal  O)$, the latter being
a (possibly irregular) octahedron. Specifically, from the rigidity of convex octahedra, it follows that there exists a function $f$ such that
$$
|u(P_1P_4)| = f(l_1,\dots,l_{12})\,,
$$
where $l_i$, $i=1,\dots ,12$, are the lengths of the twelve edges of $u(\mathcal  O)$.
In particular we set
$$
l_1:=|u(P_1P_2)|\,, \,\,
l_2:=|u(P_2P_4)|\,, \,\,
l_3:=|u(P_2P_5)|\,, \,\,
l_4:=|u(P_1P_5)|\,, \,\,
l_5:=|u(P_4P_5)|\,\,.
$$
The explicit formula of $f$ is not important.
Let $\delta_i:= l_i-1$ for
$i=1,\dots ,12$, and $\delta_0:=|u(P_1P_4)|- \sqrt{2}$.  We claim that $f$ is differentiable at
$(1,\dots,1)$.  Then
$$
f(1+\delta_1,\dots,1+\delta_{12}) = \sqrt{2} + O(|\delta|), \quad\text{ with }\delta=(\delta_1,\dots,\delta_{12})\,,
$$
which yields, in combination with Lemma \ref{lemma:tetra}, the following inequality for $\nabla u = F_1$ on the tetrahedron $T_1$:
\begin{align*}
\dist^2(F_1,SO(3)) &\leq C  \left(\sum_{i=1}^5 (l_i-1)^2 + \left(|u(P_1P_4)|- \sqrt{2}\right)^2\right) =   C \sum_{i=0}^5 \delta_i^2 \\
&= C  \sum_{i=1}^5 \delta_i^2 +
C\Big(f(1+\delta_1,\dots,1+\delta_{12}) - \sqrt{2}\Big)^2 \leq
C|\delta|^2
\end{align*}
%$$
%\dist^2(F_1,SO(3)) \leq  C  \sum_{i=0}^5 \delta_i^2 = C  \sum_{i=1}^5 \delta_i^2 +
%C\Big(f(1+\delta_1,\dots,1+\delta_{12}) - \sqrt{2}\Big)^2 \leq
%C|\delta|^2
%$$
%
for $|\delta|\leq 1$. On the other hand, by the triangle inequality, we have for $|\delta|> 1$
$$
\delta_0^2 \leq
2(f^2(1+\delta_1,\dots,1+\delta_{12}) + 2)\leq 4(l_1^2 + l_2^2 + 2)
\leq C'|\delta|^2\,.
$$
The inequality for the other $T_i$'s is completely analogous.
\par
We are left to show that $f$ is differentiable at $(1,\dots,1)$. To this end, we prove
the existence and continuity of all its partial derivatives at $(1,\dots,1)$.
By a symmetry argument, it is enough to study the existence and continuity of
$\partial_3 f$ and $\partial_4 f$ (with reference to Figure~\ref{fig:piramide}).
\par
%
%
%%%%%%%%%%%%%%%%%%%%%%%%%%%%%%%%%%%%%%%%%%%%
\begin{figure}
\centering
\psfrag{q1}{$Q_1$}
\psfrag{q2}{$Q_2$}
\psfrag{q3}{\hspace{-1mm}$Q_3 $}
\psfrag{q4}{$Q_4$}
\psfrag{q5}{$Q_5$}
\psfrag{q6}{$Q_6$}
\subfloat[]{
\includegraphics[width=.4\textwidth]{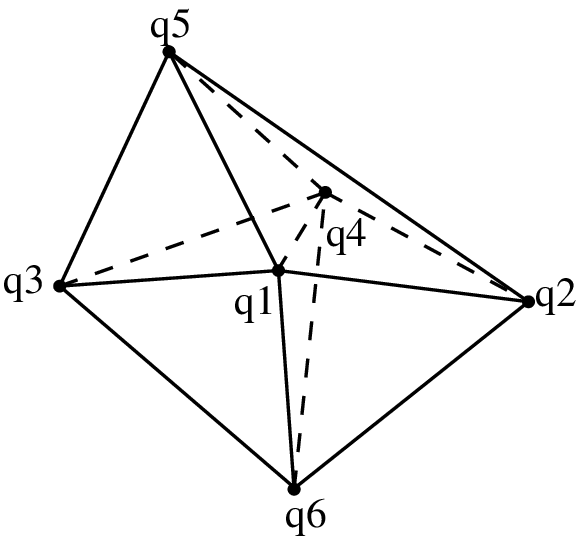}
\label{fig:piramide}
}
\hspace{.12\textwidth}
\psfrag{o}{$O$}
\psfrag{g}{$\gamma$}
\subfloat[]{
\includegraphics[width=.39\textwidth]{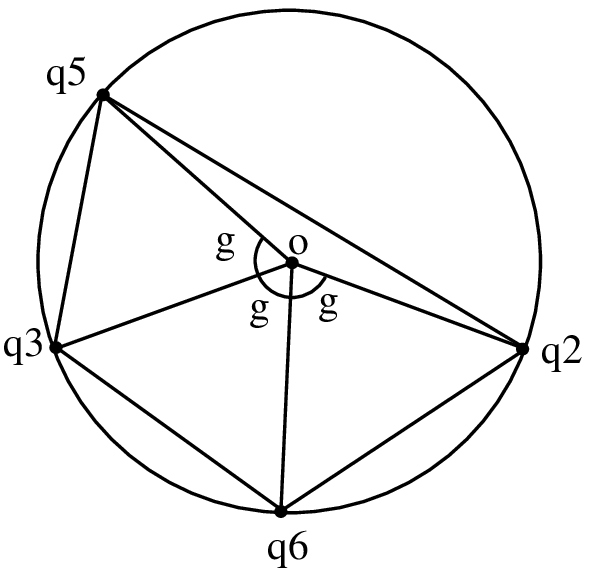}
\label{fig:proiezione}
}
\caption{(a) The image of $\O$ through a piece-wise affine map $u$ such that
$l_i=1$ for each $i\neq 3$.
(b) The projection of $u(\O)$ on the plane $p$, where $O=\Pi(Q_1)=\Pi(Q_4)$.}
\label{fig:entrambi}
\end{figure}
%
%%%%%%%%%%%%%%%%%%%%%%%%%%%%%%%%%%%%%%%%%%%%%%%%
%
%
We begin with $\partial_3 f$.
Let $u(P_i)=Q_i$ be such that $l_i=1$ for each $i\neq 3$ and $l_3\neq 1$.
Since
$f^2(1,1,l_3,1,\dots ,1)= 2-2\cos\alpha$,
where $\alpha$ is the acute angle formed by $Q_1Q_2$ and $Q_2Q_4$,
$f(1,1,l_3,1,\dots ,1)$ is a smooth function of $\alpha$ for $0<\alpha<\pi$.
Next remark that the points $Q_2,Q_3,Q_5,Q_6$ are coplanar;
let $p$ denote the plane containing them and $\Pi$ be the orthogonal projection onto $p$
(see Figure~\ref{fig:proiezione}).
Considering the projections we see that
$$
|Q_2\Pi(Q_1)|=|Q_3\Pi(Q_1)|=|Q_5\Pi(Q_1)|=|Q_6\Pi(Q_1)|=\cos\tfrac{\alpha}{2}\,.
$$
Let $\gamma$ denote the acute angle formed by $Q_2\Pi(Q_1)$ and $Q_6\Pi(Q_1)$ (which
is equal to that formed by $Q_6\Pi(Q_1)$ and $Q_3\Pi(Q_1)$ and that formed by
$Q_3\Pi(Q_1)$ and $Q_5\Pi(Q_1)$).
Then, by the cosine formula
\begin{equation}\label{formula:gamma}
\cos\gamma = 1- \frac{1}{2\cos^2\frac{\alpha}{2}}
\end{equation}
and
\begin{equation}\label{quasiformula:l3}
|Q_2Q_5|^2= 2\cos^2\tfrac{\alpha}{2}\Big(1- \cos(2\pi - 3\gamma)  \Big).
\end{equation}
Combining \eqref{formula:gamma} and \eqref{quasiformula:l3} we obtain
\begin{equation*}
l_3:= |Q_2Q_5|=
8\cos^2\tfrac{\alpha}{2} - 3 - \frac{\Big( 2\cos^2\frac{\alpha}{2} - 1\Big)^3}{\cos^4\frac{\alpha}{2}}
= 3 -\frac{1}{\cos^2\frac{\alpha}{2}} \,.
\end{equation*}
Remark that since $\alpha<\pi$, $l_3$ is a smooth function of $\alpha$.
Moreover, since the derivative $\displaystyle \frac{\d l_3}{\d\alpha}$ is not zero at the point $\alpha = \pi/2$, by the implicit function theorem it follows that $l_3$ is an invertible function of $\alpha$ in a
neighbourhood of $\alpha = \pi/2$, and its inverse is smooth in a neighbourhood of $l_3=1$.
The differentiability of $f$ with respect to $l_3$ at $(1,\dots,1)$ then follows from its smooth dependence on $\alpha$.
\par
Finally, proving the existence and continuity of $\partial_4 f$ at $(1,\dots,1)$ is equivalent to
proving that the length of  $Q_3Q_2$ is a smooth function of $l_3$ in a neighbourhood of $(1,\dots,1)$.
The latter follows equivalently to the previous argument taking into account that
$\displaystyle|Q_3Q_2|=2\cos\tfrac{\alpha}{2}$.
\end{proof}
\begin{remark}
Estimates  \eqref{rig-tetraedro} and \eqref{rig-ottaedro} are crucial in the proof of the compactness of sequences of deformations with equibounded energy, as well as in the study of the
$\Gamma$-limit and its scaling properties (see Theorem \ref{thm3} and
Proposition \ref{lowerbound}).
Indeed, as already remarked, each of the lattices introduced in Section \ref{3d} defines a tessellation
of the space into tetrahedra and octahedra.
This allows us to deduce the following lower bounds on the energy $\E_{\eps}^{\lambda}$:
\begin{align}
\label{ineq1}
\dist^2(\nabla u_\eps, SO(3)) & \leq C \E_{\eps}^{\lambda}(u_\eps) \text{ a.e. in }\Om_{k\eps}^- \\
\label{ineq2}
\dist^2\Big(\nabla u_\eps, \tfrac{\lambda}{\rho}SO(3)\Big) &  \leq C \E_{\eps}^{\lambda}(u_\eps)
\text{ a.e. in }\Om_{k\eps}^+\,,
\end{align}
for each admissible deformation $u_\eps$.
Observe, in particular, that in the case of the diamond cubic lattice the above inequalities
are obtained by first neglecting in the energy the bonds between atoms of the sublattice $\L\dc{_2}$
and then applying \eqref{rig-tetraedro} and \eqref{rig-ottaedro} on the tessellation of the space thus
defined.
\par
Inequalities \eqref{ineq1}--\eqref{ineq2} imply,
via the rigidity estimate \eqref{rigidity}, that $\nabla u_\eps$ is locally
close to $SO(3)$ in $\Om_{k\eps}^- $ and to  $ \frac{\lambda}{\rho}SO(3)$ in $\Om_{k\eps}^+ $.
\end{remark}
%
%
%%%%%%%%%%%%%%%%%%%%%%%%%%%%
%
\section{Dimension reduction and scaling properties of the $\Gamma$-limit} \label{sec:4}
In the present section we show the results in the three dimensional setting  that
were obtained in two dimensions in our previous paper \cite{LaPaSc2014}.
The proofs of these results follow the lines of those in \cite{LaPaSc2014} by application of
Lemmas \ref{lemma:tetra} and \ref{lemma:ottaedro}. We will therefore omit further details of the proofs here.
 \par
Given $R\in SO(3)$, $\rho\in(0,1]$, and $k\in\N$, we define the
minimum cost of a transition from an equilibrium in $\Om_{k,\infty}$ with $\xi_1 < 0$ to an equilibrium in $\Om_{k,\infty}$ with $\xi_1>0$ as
\begin{equation*}%\label{gammah}
\begin{split}
\ga^\lambda(\rho,k,R):=\inf\big\{ & \E_{\infty}^\lambda(v,\rho,k) \colon M>0\,,
\  v\in \A_{\rho,\infty}(\Om_{k,\infty})\,, \\
& \D v=I \ \text{for} \ x_1\in(-\infty,-M)\,,\, \ \D v=\tfrac\lambda\rho R \ \text{for} \ x_1\in(M,+\infty)
\big\}\,,
\end{split}
\end{equation*}
where $\E_{\infty}^\lambda$ is defined in \eqref{energia-infinita}.
\begin{remark}
In fact it can be proved that $\ga^\lambda(\rho,k,R)$ does not depend on the rotation $R$
(see \cite[Proposition 2.4]{LaPaSc2014}). We therefore set
$$
\ga^\lambda(\rho,k):=\ga^\lambda(\rho,k,I)\,.
$$
\end{remark}
The function $\ga^\lambda(\rho,k)$, which depends on the number of planes of atoms of
the two lattices $\L_1^-$ and $\L_\rho^+$ contained in the domain $\overline\Om_{k,\infty}$,
is in fact the relevant quantity that describes the system when $\eps$ tends to zero. More precisely,
our goal is to show that for $k$ sufficiently large, there holds
$$
\inf_{\rho\in(0,1)} \ga^\lambda(\rho,k)< \ga^\lambda(1,k)\,,
$$
i.e., the system displays dislocations.
In order to prove this,
we perform a dimension reduction with respect to the directions $\wu_2$, $\wu_3$.
To this end, for each $u_\e\in \A_{\rho,\eps}(\Om_{k\eps})$ we define the rescaled deformation
$$
\ut_\e(x):=u_\e (A_\e x)  \,,
$$
where $A_\e$ is the matrix defined by
\begin{equation*}
\begin{cases}
A_\e \wu_1 =  \wu_1\\
A_\e \wu_2 = \e \wu_2\\
A_\e \wu_3 = \e \wu_3 \,.
\end{cases}
\end{equation*}
This yields a scaling of the domain $\Om_{k\e}$ to $\Om_{k}$, which is independent of $\e$.
For fixed $\rho\in(0,1]$ and $k\in\N$ we address the question of the $\Gamma$-convergence of the sequence of functionals
$\{\I_\e\}$ defined by
$$
\I_{\e}(\tilde u_\e) := \E^\lambda_\e(u_\e,\rho,k) \quad \text{for}\ \ut_\e
\in \widetilde\A_{\rho,\e}(\Om_{k}) \,,
$$
where $ \widetilde\A_{\rho,\e}(\Om_{k})$ is the corresponding set of admissible deformations, i.e.
\bes % \label{tildeA}
\begin{split}
\widetilde\A_{\rho,\eps}(\Om_{k}):= \big\{ \ut_\eps\in C^0(\overline\Om_{k}; \R^3) \colon & \ut_\eps \ \text{piecewise affine,}\\
& \D \ut_\eps \ \text{constant on}\ \Om_{k}\cap (A_\eps^{-1}T)\ \forall\, T\in\TT_{\rho,\eps}^{(1)}\,, \\
& \det \D\ut_\eps>0 \ \text{a.e.\ in}\ \Om_{k}\,,\\
& \ut_\eps(A_\eps^{-1}\mathcal P) \ \text{is convex}\ \forall\, \mathcal P\in\TT_{\rho,\eps} \big\} \,.
\end{split}
\ees
%
%%%%%%%%%%%%%%%%%%%%%%%%%%%%%%%%%%%%%%%
%
\begin{theorem}\label{thm3}
Let $\{\ut_\e\}\subset \widetilde\A_{\rho,\eps}(\Om_{k})$ be a
sequence such that
\begin{equation*} %\label{equibounded}
\limsup_{\e\to 0^{+}} \I_{\e}(\ut_\e) \leq C \,.
\end{equation*}
Then there exists a subsequence (not relabeled) such that
\begin{equation*}
\D \ut_\e A_\e^{-1}\weakst (\partial_{\wu_1} \ut \, | \, d_2 \, | \, d_3) \quad \text{weakly* in}\ L^{\infty}(\Om_k;\M^{3\times 3})\,,
\end{equation*}
where the functions $\ut \in W^{1,\infty}(\Om_k;\R^3)$, $d_2, d_3\in L^{\infty}(\Om_k;\R^3)$
are independent of $\wu_2$ and $\wu_3$, i.e.,
$\partial_{\wu_i}\ut = \partial_{\wu_i} d_2 =  \partial_{\wu_i} d_3= 0$, for $i=2,3$.
Moreover,
\begin{equation*} %\label{charact}
 (\partial_{\wu_1} \ut \, | \, d_2 \, | \, d_3)
 \in
\begin{cases}
\co(SO(3))     &  \text{a.e.\ in}\ \Om_k^-\,,\\
\co(\frac\lambda\rho SO(3))   &  \text{a.e.\ in}\ \Om_k^+\,.
\end{cases}
\end{equation*}
%%%%%%%%%%%%%%%%%%%%%%%%%%%%%%%%%%%%%%%%%%
The sequence of functionals $\{  \I_{\e}\}$
$\Gamma$-converges, as $\e\to 0^{+}$, to the functional
\begin{equation*}%\label{eqthm3}
\I(u)=
\begin{cases}
\ga^\lambda(\rho,k)   &  \text{if}\ u\in\A\,,\\
+\infty   &  \text{otherwise,}
\end{cases}
\end{equation*}
with respect to the weak* convergence in $W^{1,\infty}(\Om_k;\R^3)$, where
\begin{equation*}
%\label{gammadomain}
\begin{split}
\A:=\big\{
u\in W^{1,\infty}(\Om_k;\R^3)\colon &
\partial_{\wu_2} u  = \partial_{\wu_3} u =0 \ \text{a.e.\ in}\ \Om_k\,, \\
& |\partial_{\wu_1} u |\leq 1 \ \text{a.e.\ in}\ \Om_k^- \,,
|\partial_{\wu_1} u |\leq \tfrac\lambda\rho \ \text{a.e.\ in}\ \Om_k^+
\big\} \,.
\end{split}
\end{equation*}
\end{theorem}
%
%%%%%%%%%%%%%%%%%%%%%%%%%%%%
%
\begin{remark} \label{rmk5}
As in the two-dimensional case, the limit of a sequence of discrete deformations with equibounded energy
does not depend on the triangulation chosen for the octahedra, see \cite[Remark 3.4]{LaPaSc2014}.
Moreover, the $\Gamma$-limit does not depend on the choice of the triangulation $\TT_{\rho,\eps}^{(1)}$ in \eqref{ad-3d},
since its formula only depends on the discrete values of the deformation, and not on its extension to the three-dimensional continuum.  Similarly, it does not depend on the choice of the tessellation of
$\L\dc{}$.
\end{remark}
The next two results characterise the behaviour of the $\Gamma$-limit in the
dislocation-free case and in the case when dislocations are present.
\begin{proposition}[Estimate in the defect-free case, $\rho=1$]\label{lowerbound}
There exist $C_1,C_2>0$ such that for every $k\in\N$
\begin{equation*}
C_1 k^3 \leq \ga^\lambda(1,k) \leq C_2 k^3 \,.
\end{equation*}
\end{proposition}
\begin{proposition}[Estimate for $\rho=\lambda$] \label{prop:lambda}
There exist positive constants $C_1',C_2'$ such that for every $k$
\bes
C_1' k^2\le \gamma^\lambda(\lambda,k)\le C_2' k^2 \,.
\ees
\end{proposition}
\begin{remark}
The proof of Proposition \ref{lowerbound} is a generalisation of its two-dimensional
counterpart (see  \cite[Proposition 2.5]{LaPaSc2014}) given the rigidity estimates of the previous section.
In contrast,
the proof of Proposition \ref{prop:lambda} is straightforward: it follows
by testing  $\E_{\infty}^\lambda$ on the identical deformation $v(x)=x$ and
taking into account that
each interfacial atom has a number of bonds that is uniformly bounded in $k$.
\end{remark}
\begin{remark}\label{finalrem}
Propositions \ref{lowerbound}--\ref{prop:lambda} in combination with Theorem \ref{thm3}
prove that dislocations
are energetically preferred if the thickness of the nanowire modeled by $k$ is sufficiently large.
\end{remark}
%
%
%\begin{remark}
%Lemmas \ref{lemma:tetra} and \ref{lemma:ottaedro} are used in the proof of both Theorem \ref{thm3} and
%Proposition \ref{lowerbound}.
%\end{remark}
%
%
\section*{Acknowledgements} \noindent
This work was partially supported by the DFG grant SCHL 1706/2-1.
The research of G.L.\ was supported by the ERC grant No.\ 290888.

\end{document}